\documentclass[11pt]{amsart}

\usepackage[english]{babel}
\usepackage[T1]{fontenc}

\usepackage[utf8]{inputenc}
\usepackage{amsmath,amssymb,mathrsfs}
\usepackage{amscd}
\usepackage{pgf,tikz}

\usepackage{hyperref}

\numberwithin{equation}{section}
\newtheorem{theo}{Theorem}[section]
\newtheorem{cor}{Corollary}[section]
\newtheorem{cnj}{Conjecture}[section]
\newtheorem{prop}{Proposition}[section]
\newtheorem{lem}{Lemma}[section]

\theoremstyle{remark} 
\newtheorem{rem}{\textbf{Remark}}[section]
\numberwithin{figure}{section}

\newcommand{\bn}{{\mathbb{N}}}
\newcommand{\bz}{{\mathbb{Z}}}
\newcommand{\brrr}{{\mathbb{R}}}
\newcommand{\bc}{{\mathbb{C}}}

\newcommand{\sqq}{{\textbf{\textup{S}}_q}}
\newcommand{\sd}{{\textbf{\textup{S}}_2}}
\newcommand{\sinf}{{\textbf{\textup{S}}_\infty}}
\newcommand{\one}{{\bf 1}}

\newcommand{\im}{\operatorname{Im}}
\newcommand{\re}{\operatorname{Re}}

\newcommand{\xp}{x_\perp}

\begin{document}

\title[A criterion for the existence of non-real eigenvalues]
{A criterion for the existence of non-real eigenvalues for a Dirac operator}

\author{Diomba \textsc{Sambou}}

\address{Departamento de Matem\'aticas, Facultad de Matem\'aticas,
 Pontificia Universidad Cat\'olica de Chile, Vicu\~na Mackenna 4860, 
 Santiago de Chile}
 
\email{disambou@mat.uc.cl}

\thanks{This work is partially supported by the Chilean 
Program \textit{N\'ucleo Milenio de F\'isica Matem\'atica
RC$120002$}. The author wishes to express his thanks to R. 
Tiedra de Aldecoa for several helpful comments during the
preparation of the paper, and to R. L. Frank for bringing 
to his attention the reference \cite{cuen}.}



%\date{January 1, 2004}

\begin{abstract}
The aim of this work is to explore the discrete spectrum 
generated by complex perturbations in $ L^{2}(\brrr^3,\bc^4)$ 
of the $3d$ Dirac operator $\alpha \cdot 
(-i\nabla - \textbf{A}) + m \beta$ with variable magnetic 
field. Here, $\alpha := (\alpha_1,\alpha_2,\alpha_3)$ and 
$\beta$ are $4 \times 4$ Dirac matrices, and $m > 0$ is 
the mass of a particle. We give a simple criterion for the 
potentials to generate discrete spectrum near $\pm m$. In 
case of creation of non-real eigenvalues, this criterion 
gives also their location.
\end{abstract}

\maketitle

\section{Introduction}\label{s1}

In this paper, we consider a Dirac operator
$D_m(b,V)$ defined as follows. Denoting 
$x = (x_1,x_2,x_3)$ the usual variables of 
$\brrr^3$, let 
\begin{equation}
\textbf{B} = (0,0,b)
\end{equation} 
be a nice scalar magnetic field with constant 
direction such that $b= b(x_1,x_2)$ 
is \textit{an admissible magnetic field}. That is, 
there exists a constant $b_{0} > 0$ satisfying
\begin{equation} 
b(x_1,x_2) = b_{0} + \tilde{b}(x_1,x_2),
\end{equation} 
where $\tilde{b}$ is a function such that 
the Poisson equation 
$\Delta \tilde{\varphi} = \tilde{b}$
admits a solution $\tilde{\varphi} 
\in C^{2}(\mathbb{R}^{2})$ verifying 
$\sup_{(x_1,x_2) \in \mathbb{R}^{2}} \vert 
D^{\alpha} \tilde{\varphi}(x_1,x_2) \vert < 
\infty$, $\alpha \in \mathbb{N}^{2}$, $\vert 
\alpha \vert \leq 2$. Define on $\brrr^2$ 
the function $\varphi_0$ by $\varphi_{0} 
(x_1,x_2):= \frac{1}{4}b_{0}(x_1^2 + x_2^2)$
and set
\begin{equation}  
\varphi (x_1,x_2) := \varphi_{0} (x_1,x_2) + 
\tilde{\varphi} (x_1,x_2).
\end{equation}
We obtain a magnetic potential 
$\textbf{A} : \brrr^3 \longrightarrow \brrr^3$ 
generating the magnetic field $\textbf{B}$ 
(i.e. $\textbf{B} = \text{curl} \hspace{0.6mm} 
\textbf{A}$) by setting 
\begin{equation}
\begin{split}
A_{1}(x_1,x_2,x_3) & = A_{1}(x_1,x_2) = - \partial_{x_{2}} 
\varphi (x_1,x_2), \\
A_{2}(x_1,x_2,x_3) & = A_{2}(x_1,x_2) = \partial_{x_{1}}
\varphi (x_1,x_2), \\
A_{3}(x_1,x_2,x_3) & = 0.
\end{split}
\end{equation}
Then, for a $4 \times 
4$ complex matrix $V = \big\lbrace V_{\ell 
k}(x) \big\rbrace_{\ell,k = 1}^4$, the Dirac 
operator $D_m(b,V)$ acting on $L^{2}(\brrr^3) 
:= L^{2}(\brrr^3,\bc^4)$ is defined by 
\begin{equation}\label{eq1,7}
D_m(b,V) := \alpha \cdot (-i\nabla - \textbf{A}) 
+ m \beta  + V,
\end{equation}
where $m > 0$ is the mass of a particle. Here,
$\alpha = (\alpha_1,\alpha_2,\alpha_3)$ and 
$\beta$ are the Dirac matrices defined by the 
following relations:
\begin{equation}\label{eq1,8}
\alpha_{j}\alpha_{k} + \alpha_{k}\alpha_{j} = 
2\delta_{jk}\textbf{1}, \quad \alpha_{j} 
\beta + \beta\alpha_{j} = \textbf{0}, 
\quad \beta^{2} = \textbf{1}, \quad j, k 
\in \lbrace 1,2,3 \rbrace,
\end{equation}
$\delta_{jk}$ being the Kronecker symbol 
defined by $\delta_{jk} = 1$ if $j = k$ and 
$\delta_{jk} = 0$ otherwise, (see e.g. 
the book \cite[Appendix of Chapter 1]{tha} 
for other possible representations).
\\

\noindent
For $V = 0$, it is known that the spectrum of 
$D_m(b,0)$ is $(-\infty,-m] \cup [m,+\infty)$ 
(see for instance \cite{tda,diom}). Throughout 
this paper, we assume that $V$ satisfies
\\

\noindent
\textbf{Assumption 1.1.} $V_{\ell k}(x) \in
\bc$ for $1 \leq \ell,k \leq 4$ with
\begin{equation}\label{eq1,17}
\begin{split}
& \bullet \hspace{0.6mm}  0 \not\equiv V_{\ell k} \in 
L^{\infty} (\brrr^3), \, \vert V_{\ell k}(x) 
\vert \lesssim F_{\perp} (x_1,x_2) \hspace{0.5mm} 
G(x_3), \\
& \bullet \hspace{0.6mm} F_{\perp} \in \bigl( 
L^\frac{q}{2} \cap L^\infty \bigr) \big( 
\brrr^{2},\brrr_{+}^{\ast} \big) \hspace{0.5mm} 
\textup{for some $q \geq 4$}, \\
& \bullet \hspace{0.6mm} 0 < G(x_3) \lesssim 
\langle x_3 \rangle^{-\beta}, \beta > 3, 
\hspace{0.4mm} \textup{where} \hspace{0.4mm} 
\langle y \rangle := \sqrt{1 + \vert y 
\vert^2} \hspace{1mm} \textup{for} 
\hspace{1mm} y \in \brrr^d.
\end{split}
\end{equation}

\begin{rem}
Assumption $1.1$ is naturally satisfied 
by matrix-valued perturbations $V : \brrr^3 
\rightarrow \bc^4$ (not necessarily 
Hermitian) such that
\begin{equation}\label{eqe}
\vert V_{\ell k}(x) \vert \lesssim 
\langle (x_1,x_2) \rangle^{-\beta_{\perp}} 
\langle x_3 \rangle^{-\beta}, 
\quad \beta_{\perp} > 0, \quad  \beta > 3, 
\quad 1 \leq \ell,k \leq 4.
\end{equation}
We also have the matrix-valued perturbations 
$V : \brrr^3 \rightarrow \bc^4$ (not 
necessarily Hermitian) such that
\begin{equation}\label{eq1,191}
\vert V_{\ell k}(x) \vert \lesssim \langle 
x \rangle^{-\gamma}, 
\quad \gamma > 3, \quad 1 \leq \ell,k \leq 4.
\end{equation}
Indeed, it follows from \eqref{eq1,191} that 
\eqref{eqe} holds 
%$
%\vert V_{\ell k}(x) \vert \lesssim \langle 
%(x_1,x_2) \rangle^{-\beta_{\perp}} \langle 
%x_3 \rangle^{-\beta}, \quad 1 \leq \ell,k 
%\leq 4
%$
with any $\beta \in (3,\gamma)$ and 
$\beta_{\perp} = \gamma - \beta > 0$.
\end{rem}

\noindent
Since we will deal with non-self-adjoint 
operators, it is useful to precise the 
notion used of discrete and essential 
spectrum of an operator acting on a 
separable Hilbert space $\mathscr{H}$. 
%\begin{fe}
Consider $S$ a closed such operator. Let 
$\mu$ be an isolated 
point of ${\rm sp}\,(S)$, and $\mathscr{C}$ 
be a small positively oriented circle 
centred at $\mu$, containing $\mu$ as the 
only point of ${\rm sp}\,(S)$. The point 
$\mu$ is said to be a discrete 
eigenvalue of $S$ if it's algebraic 
multiplicity
\begin{equation}\label{eq1,19}
\text{mult}(\mu) := \text{rank} \left( 
\frac{1}{2i\pi} \int_{\mathscr{C}} 
(S - z)^{-1}dz \right)
\end{equation}
is finite. The discrete spectrum of 
$S$ is then defined by 
\begin{equation}\label{eq1,20}
{\rm sp}\,_{\text{\textbf{\textup{disc}}}}(S) 
:= \big\lbrace \mu \in {\rm sp}\,(S) : \mu 
\hspace*{0.1cm} \textup{is a discrete 
eigenvalue of $S$} \big\rbrace.
\end{equation}
%\end{fe}
Notice that the geometric multiplicity 
$\text{dim} \big( \text{Ker}(S - \mu) \big)$ 
of $\mu$ is such that $\text{dim} \big( 
\text{Ker}(S - \mu) \big) \leq 
\text{mult}(\mu)$. Equality holds if $S$ is 
self-adjoint. 
The essential spectrum of $S$ is defined by
\begin{equation}\label{eq1,21}
{\rm sp}\,_{\text{\textbf{\textup{ess}}}}(S) 
:= \big\lbrace \mu \in \bc : \textup{$S - \mu$ 
is not a Fredholm operator} \big\rbrace.
\end{equation} 
It's a closed subset of ${\rm sp}\,(S)$.
\\

\noindent
Under Assumption $1.1$, we show (see 
Subsection \ref{s4,1}) that $V$ is relatively 
compact with respect to $D_m(b,0)$. Therefore, 
according to the Weyl criterion on the 
invariance of the essential spectrum, we have 
\begin{equation}
\begin{split}
{\rm sp}\,_{\text{\textbf{ess}}} & \big( 
D_m(b,V) \big) = {\rm sp}\,_{\text{\textbf{ess}}} 
\big( D_m(b,0) \big) \\
& = {\rm sp}\, \big( D_m(b,0) 
\big) = (-\infty,-m] \cup [m,+\infty).
\end{split}
\end{equation}
However, $V$ may generate complex eigenvalues 
(or discrete spectrum) that can only accumulate 
on $(-\infty,-m] \cup [m,+\infty)$ 
(see \cite[Theorem 2.1, p. 373]{goh}). The 
situation near $\pm m$ is the most interesting 
since they play the role of spectral thresholds 
of this spectrum. 
For the quantum Hamiltonians, many studies on 
the distribution of the discrete spectrum near 
the essential spectrum have been done for 
self-adjoint perturbations, see for instance 
\cite[Chap. 11-12]{iv}, \cite{pus,sob,tam,roz,
diom,tda} and the references therein.
Recently, there has been an increasing 
interest in the spectral theory of 
non-self-adjoint differential operators. 
We quote for instance the papers 
\cite{wan,fra,bru1,bor,dem,demu,han,gol,dio}, 
see also the references therein. In most of 
these papers, (complex) eigenvalues estimates 
or Lieb-Thirring type inequalities are 
established. However, the problem of the 
existence and the localisation of the 
complex eigenvalues near the essential 
spectrum of the operators is not addressed.
We can think that this is probably due to the 
technical difficulties caused by the 
non-self-adjoint aspect of the perturbation.
By the same time, there are few results 
concerning non-self-adjoint Dirac operators, 
\cite{syr,syr1,cue,dub,cuen}. In this article, 
we will examine the problem 
of \textit{the existence}, \textit{the 
distribution} and \textit{the localisation} 
of the non-real eigenvalues of the Dirac 
operator $D_m(b,V)$ near $\pm m$. The case 
of the non-self-adjoint Laplacian 
$-\Delta + V(x)$ in $L^2(\brrr^n)$, 
$n \geq 2$, \textit{near the origin}, is 
studied by Wang in \cite{wan}. In particular, 
he proves that for slowly decaying potentials,
$0$ is the only possible accumulation point 
of de complex eigenvalues and if $V(x)$ decays
more rapidly the $\vert x \vert^{-2}$, then 
there are no clusters of eigenvalues near the
points of $[0,+\infty)$.
Actually, in Assumption $1.1$, the 
condition
\begin{equation}\label{eq1,22}
0 < G(x_3) \lesssim 
\langle x_3 \rangle^{-\beta}, \quad \beta 
> 3, \quad x_3 \in \brrr,
\end{equation}
is required in such a way we include 
perturbations decaying polynomially (as 
$\vert x_3 \vert \longrightarrow + \infty$) 
along the direction of the magnetic field. 
In more restrictive setting, if we replace 
\eqref{eq1,22} by perturbations decaying 
exponentially along the direction of the 
magnetic field, i.e. satisfying
\begin{equation}\label{eq1,23}
0 < G(x_3) \lesssim 
e^{-\beta \langle x_3 \rangle}, \quad \beta 
> 0, \quad x_3 \in \brrr,
\end{equation}
then our third main result (Theorem 
\ref{t2,4}) can be improved to get non-real
eigenvalues asymptotic behaviours near 
$\pm m$. However, this topic is beyond 
these notes in the sense that it requires 
the use of resonance approach, by defining in 
Riemann surfaces the resonances of the 
non-self-adjoint operator $D_m(b,V)$ near 
$\pm m$, and it will be considered 
elsewhere. Here, we extend and generalize 
to non-self-adjoint matrix case the methods 
of \cite{diom,bon}. And, the problem
studied is different.
Moreover, due to the structure of the 
essential spectrum of the Dirac operator 
considered here (symmetric with respect to 
the origin), technical difficulties appear.
In particular, these difficulties are 
underlying to the choice of the complex 
square root and the parametrization 
of the discrete eigenvalues in a neighbourhood 
of $\pm m$ (see \eqref{eq2,7}, Remarks \ref{r2,1} 
and \ref{r5,1}). To prove our main 
results, we reduce the study of the 
complex eigenvalues to 
the investigation of zeros of holomorphic
functions. This allows us to essentially use
complex analysis methods to solve our
problem. Firstly, we obtain sharp upper 
bounds on the number of complex
eigenvalues in small annulus near 
$\pm m$ (see Theorem \ref{t2,1}). Secondly, 
under appropriate hypothesis, we prove the 
absence of non-real eigenvalues in certain 
sectors adjoining $\pm m$ (see Theorem \ref{t2,3}). 
By this way, we derive from Theorem \ref{t2,3} 
a relation between the properties of the 
perturbation $V$ and the finiteness of the 
number of non-real eigenvalues of $D_m(b,V)$ 
near $\pm m$ (see Corollary \ref{c1}). Under 
additional conditions, we prove lower bounds 
implying the existence of non-real eigenvalues 
near $\pm m$ (see Theorem \ref{t2,4}). 
In more general setting, we conjecture a 
criterion of non-accumulation of the discrete 
spectrum of $D_m(b,V)$ near $\pm m$ 
(see Conjecture \ref{cj2,1}). This conjecture 
is in the spirit of the Behrndt conjecture 
\cite[Open problem]{ber} on Sturm-Liouville 
operators. More precisely, he says the 
following: there exists non-real eigenvalues 
of singular indefinite Sturm-Liouville 
operators accumulate to the real axis 
whenever the eigenvalues of the corresponding
definite Sturm-Liouville operator accumulate
to the bottom of the essential spectrum 
from below.
\\
%\subsubsection*{Organization of the paper}

\noindent
The paper is organized as follows.
We present our main results in Section \ref{s2}. 
In Section \ref{s3}, we estimate the Schatten-von 
Neumann norms (defined in Appendix A) of the
(weighted) resolvent of $D_m(b,0)$. We also reduce 
the study of the discrete spectrum to that of  
zeros of holomorphic functions. In Section 
\ref{s4}, we give a suitable decomposition 
of the (weighted) resolvent of $D_m(b,0)$. Section
\ref{s5} is devoted to the proofs of the main results. 
Appendix A is a summary on basic properties 
of the Schatten-von Neumann classes. In Appendix B,
we briefly recall the notion of the index of a finite 
meromorphic operator-valued function along a 
positive oriented contour.

\section{Formulation of the main results}\label{s2}

In order to state our results, some 
additional notations are needed. Let $p = p(b)$ 
be the spectral projection of $L^2 ( \brrr^2)$ 
onto the (infinite-dimensional) kernel of
\begin{equation}\label{eqh}
H_\perp^- := (-i\partial_{x_1} - A_1)^2 + (-i\partial_{x_1} 
- A_2)^2 - b,
\end{equation}
(see \cite[Subsection 2.2]{raik}).
For a complex $4 \times 4$ matrix $M = M(x)$, 
$x \in \brrr^3$, $\vert M \vert$ 
defines the multiplication  operator in 
$L^2(\brrr^3)$ by the matrix $\sqrt{M^\ast M}$.
Let $\textbf{\textup{V}}_{\pm m}$ be the 
multiplication operators by the functions 
\begin{equation}\label{eq2,2}
\begin{split}
\displaystyle \textbf{\textup{V}}_m(x_1,x_2) & = 
\frac{1}{2} \int_\brrr v_{11} (x_1,x_2,x_3) 
dx_3, \\ 
\displaystyle \textbf{\textup{V}}_{-m}
(x_1,x_2) & = \frac{1}{2} \int_\brrr v_{33}
(x_1,x_2,x_3) dx_3,
\end{split}
\end{equation}
where $v_{\ell k}$, $1 \leq \ell,k \leq 4$, 
are the coefficients of the matrix $\vert V 
\vert$. Clearly, Assumption $1.1$ implies that 
\begin{equation}
0 \leq \textbf{\textup{V}}_{\pm m}(x_1,x_2)
\lesssim \sqrt{F_\perp(x_1,x_2)},
\end{equation} 
since $F_\perp$ is bounded. This together with
\cite[Lemma 2.4]{raik} give that the 
self-adjoint Toeplitz operators 
$p \textbf{\textup{V}}_{\pm m} p$ are compacts. 
Defining 
\begin{equation}
\bc_\pm := \big\lbrace z \in \bc : \pm 
\im(z) > 0 \big \rbrace,
\end{equation} 
we will adopt the following choice of 
the complex square root
\begin{equation}\label{eq2,7}
\bc \setminus (-\infty,0] \overset{\sqrt{\cdot}}
{\longrightarrow} \bc_+.
\end{equation}
Let $\eta$ be a fixed constant such that 
$0 < \eta < m$. For $\widetilde{m} 
\in \lbrace \pm m \rbrace$, we set
\begin{equation}\label{eq2,8}
\mathcal{D}_{\widetilde{m}}^\pm (\eta) 
:= \big\lbrace z \in \bc_\pm : 0 < \vert z - 
\widetilde{m} \vert < \eta \big\rbrace.
\end{equation}
If $0 < \gamma < 1$ and $0 < \epsilon < 
\min \left( \gamma,\frac{\eta(1 - \gamma)}{2m} 
\right)$, we define the domains
\begin{equation}\label{eq2,10}
\mathcal{D}_\pm^\ast (\epsilon) 
:= \big\lbrace k \in \bc_\pm : 0 < \vert k 
\vert < \epsilon : \re(k) > 0 \big\rbrace.
\end{equation}
Note that $0 < \epsilon < \eta/m$. Actually, 
the singularities of the resolvent 
of $D_m(b,0)$ at $\pm m$  are induced by 
those of the resolvent of the one-dimensionnal 
Laplacian $-\partial_{x_3}^2$ at zero 
(see \eqref{eq4,2}-\eqref{eq4,4}). Therefore,
the complex eigenvalues $z$ of $D_m(b,V)$ near 
$\pm m$ are naturally parametrized by 
\begin{equation}\label{eq2,11}
\begin{split}
\bc \setminus & {\rm sp}\, \big( D_m(b,0) \big) 
\ni z = z_{\pm m}(k) := \frac{\pm m (1 + k^2)}{1 
- k^2} \\
& \Leftrightarrow k^2 = \frac{z \mp m}{z 
\pm m} \in \bc \setminus [0,+\infty).
\end{split}
\end{equation}

\begin{rem}\label{r2,1}
\textbf{(i)} Observe that 
\begin{equation}\label{eq2,111}
\bc \setminus {\rm sp}\,\big( D_m(b,0) \big) \ni z 
\longmapsto \Psi_\pm (z) = \frac{z \mp m}{z \pm m}
\in \bc \setminus [0,+\infty)
\end{equation}
are Möbius transformations with inverses  
$\Psi_\pm^{-1}(\lambda) = \frac{\pm m (1 + 
\lambda)}{1 - \lambda}$.

\textbf{(ii)} For any $k \in \bc 
\setminus \lbrace \pm 1 \rbrace$, we have
\begin{equation}\label{eq2,12}
z_{\pm m}(k) = \pm m \pm 
\frac{2mk^2}{1 - k^2} \quad \text{and} \quad 
\im \big( z_{\pm m}(k) \big) = 
\pm \frac{2m \im(k^2)}{\vert 1 - k^2 \vert^2}.
\end{equation}

\textbf{(iii)} According to \eqref{eq2,12},
$\pm \im \big( z_m(k) \big) > 0$ if and 
only if $\pm \im(k^2) > 0$. Then, it is easy 
to check that any $z_m(k) \in \bc_\pm$ is 
respectively associated to a unique 
$k \in \bc_\pm \cap \big\lbrace k \in \bc: 
\re(k) > 0 \big\rbrace$. Moreover,
\begin{equation}\label{eq2,13}
\textup{$z_m(k) \in \mathcal{D}_{m}^\pm (\eta)$ 
\textup{whenever} $k \in \mathcal{D}_\pm^\ast 
(\epsilon)$}.
\end{equation}

\textbf{(iv)} Similarly, according to 
\eqref{eq2,12}, we have $\pm \im \big( z_{-m}(k) 
\big) > 0$ if and only if $\mp \im(k^2) > 0$. 
Then, any 
$z_{-m}(k) \in \bc_\pm$ is respectively 
associated to a unique $k \in \bc_\mp \cap 
\big\lbrace k \in \bc: \re(k) > 0 \big\rbrace$. 
Furthermore,
\begin{equation}\label{eq2,14}
\textup{$z_{-m}(k) \in \mathcal{D}_{-m}^\pm (\eta)$ 
\textup{whenever} $k \in \mathcal{D}_\mp^\ast 
(\epsilon)$}.
\end{equation}
\end{rem}

\noindent
In the sequel, to simplify the notations,
we set
\begin{equation}\label{eq2,140}
\begin{split}
& {\rm sp}\,_{\textup{\textbf{disc}}}^+ \big( 
D_m(b,V) \big) := {\rm sp}\,_{\textup{\textbf{disc}}} 
\big( D_m(b,V) \big) \cap \mathcal{D}_{\pm m}^+ 
(\eta), \\
& {\rm sp}\,_{\textup{\textbf{disc}}}^- \big( 
D_m(b,V) \big) := {\rm sp}\,_{\textup{\textbf{disc}}} 
\big( D_m(b,V) \big) \cap \mathcal{D}_{\pm m}^- 
(\eta).
\end{split}
\end{equation}

\noindent
We can now state our first main result. 

\begin{theo}[\textbf{Upper bound}]\label{t2,1}
Assume that Assumption 1.1 holds with $m > \Vert V \Vert $ small enough. 
Then, we have
\begin{equation}\label{eq2,15}
\begin{split}
\small{\displaystyle \sum_{\substack{z_{\pm m}(k) 
\hspace{0.5mm} \in \hspace{0.5mm} 
{\rm sp}\,_{\textup{\textbf{disc}}}^+ \big( D_m(b,V) 
\big) \\ k \hspace{0.5mm} \in \hspace{0.5mm} 
\Delta_\pm}}} & 
\small{\textup{mult} \big( z_{\pm m}(k) \big)
\quad + \sum_{\substack{z_{\pm m}(k) 
\hspace{0.5mm} \in \hspace{0.5mm} 
{\rm sp}\,_{\textup{\textbf{disc}}}^- \big( D_m(b,V) 
\big) \\ k \hspace{0.5mm} \in \hspace{0.5mm} 
\Delta_\mp}} 
 \textup{mult} \big( z_{\pm m}(k) \big)} \\
& \small{= \mathcal{O} \Big( 
\textup{Tr} \hspace{0.4mm} \one_{(r,\infty)} \big( 
p \textbf{\textup{V}}_{\pm m} 
p \big) \vert \ln r \vert \Big) + \mathcal{O}(1),}
\end{split}
\end{equation}
for some $r_{0} > 0$ small enough and any $r > 0$ with 
$r < r_{0} < \sqrt{\frac{3}{2}}r$,  
where $\textup{mult} \big( z_{\pm m}(k) \big)$ is 
defined by \eqref{eq1,19} and
$$
\Delta_\pm := \left\lbrace r < \vert k 
\vert < 2r : \vert \re(k) \vert > \sqrt{\nu} : 
\vert \im(k) \vert > \sqrt{\nu} : 0 < \nu \ll 1\right\rbrace \cap 
\mathcal{D}_\pm^\ast(\epsilon).
$$
\end{theo}

\bigskip

\noindent
In order to state the rest of the results, 
we put some restrictions on $V$.
\\

\noindent
\textbf{Assumption 2.1.} $V$ satisfies 
Assumption $1.1$ with
\begin{equation}\label{eq2,17}
V = \Phi W, \hspace{0.1cm} \Phi \in \bc 
\setminus \brrr, \hspace{0.1cm} \text{and} 
\hspace{0.1cm} W = \big\lbrace W_{\ell k} 
(x) \big\rbrace_{\ell,k=1}^4 \hspace{0.1cm} 
\text{is Hermitian}.
\end{equation}
The potential $W$ will be said to be 
of definite sign if $\pm W(x) \geq 0$ 
for any $x \in \brrr^3$.
Let $J := sign(W)$ denote the matrix sign of $W$.
Without loss of generality, we will say that $W$ 
is of definite sign $J = \pm$. For any 
$\delta > 0$, we set
\begin{equation}\label{eq2,18}
\mathcal{C}_\delta(J) := \big\lbrace k \in \bc : 
- \delta J \im(k) \leq \vert \re(k) \vert 
\big\rbrace, \quad J = \pm.
\end{equation}

\begin{rem}\label{r2,2}
For $W \geq 0$ and $\pm \sin({\rm Arg}\, \Phi) 
> 0$, the non-real eigenvalues $z$ of $D_m(b,V)$ 
verify $\pm \im(z) > 0$. Then, according 
to \textbf{(iii)}-\textbf{(iv)} of Remark 
\ref{r2,1}, they satisfy near $\pm m$:

\textbf{(i)} $z = z_{\pm m}(k) = \frac{\pm m 
(1 + k^2)}{1 - k^2} \in \mathcal{D}_{\pm m}^+ 
(\eta)$, $k \in \mathcal{D}_\pm^\ast(\epsilon)$ 
if $\sin({\rm Arg}\, \Phi) > 0$,

\textbf{(ii)} $z = z_{\pm m}(k) = \frac{\pm m 
(1 + k^2)}{1 - k^2} \in \mathcal{D}_{\pm m}^- 
(\eta)$, $k \in \mathcal{D}_\mp^\ast(\epsilon)$
if $\sin({\rm Arg}\, \Phi) < 0$.
\end{rem}

\begin{theo}[\textbf{Absence of non-real eigenvalues}]\label{t2,3}
Assume that $V$ satisfies Assumptions 1.1 and
2.1 with $W \geq 0$. Then, for any $\delta > 0$ 
small enough, there exists $\varepsilon_0 > 0$ 
such that for any $0 < \varepsilon \leq \varepsilon_0$, 
$D_m(b,\varepsilon V)$ has no non-real eigenvalues in
\begin{equation}\label{e2,27}
\small{\left\lbrace z = z_{\pm m}(k) \in
\begin{cases}
 \mathcal{D}_{\pm m}^+(\eta) :
 k \in \Phi \mathcal{C}_\delta(J) \cap 
 \mathcal{D}_\pm^\ast (\epsilon) & \text{for } 
 {\rm Arg}\, \Phi \in (0,\pi), \\
 \mathcal{D}_{\pm m}^-(\eta) :
 k \in -\Phi \mathcal{C}_\delta(J) \cap 
 \mathcal{D}_\mp^\ast(\epsilon) & \text{for } 
 {\rm Arg}\, \Phi \in -(0,\pi),
 \end{cases}
\right\rbrace}
\end{equation}
for $0 < \vert k \vert \ll 1$.
\end{theo}

\noindent
For $\Omega$ a small pointed 
neighbourhood of $\widetilde{m} 
\in \lbrace \pm m \rbrace$, let us 
introduce the counting function of 
complex eigenvalues of the operator 
$D_m(b,V)$ lying in $\Omega$, 
taking into account the multiplicity:
\begin{equation}\label{defNq}
\begin{split}
 {\mathcal N}_{\widetilde{m}} & \big( \big( 
 D_m(b,V) \big),\Omega \big) := \\
  \# 
 & \big\lbrace z = z_{\widetilde{m}}(k) \in 
 {\rm sp}\,_
 {\textup{\textbf{disc}}} \big( D_m(b,V) 
\big) \cap \bc_\pm \cap \Omega: 0 < \vert k \vert \ll 1
\big\rbrace.
\end{split}
\end{equation}
As an immediate consequence of Theorem 
\ref{t2,3}, we have the following

\begin{cor}[\textbf{Non-accumulation of non-real
eigenvalues}]\label{c1}
Let the assumptions of Theorem \ref{t2,3} 
hold. Then, for any $0 < \varepsilon \leq 
\varepsilon_0$ and any domain $\Omega$ as
above, we have
\begin{equation}
\begin{cases}
 {\mathcal N}_{m} \big( \big( D_m(b,\varepsilon V) 
 \big),\Omega \big) < \infty & \text{for } 
 {\rm Arg}\, \Phi \in \pm \big( 0,\frac{\pi}{2} 
 \big), \\
 {\mathcal N}_{-m} \big( \big( D_m(b,\varepsilon V) 
 \big),\Omega \big) < \infty & \text{for } 
 {\rm Arg}\, \Phi \in \pm \big( \frac{\pi}{2},\pi 
 \big).
 \end{cases}
\end{equation}
\end{cor}

\noindent
Indeed, near $m$, for ${\rm Arg}\, \Phi \in 
\pm \big( 0,\frac{\pi}{2} \big)$ and $\delta$
small enough, we have respectively $\pm \Phi 
\mathcal{C}_\delta(J) \cap \mathcal{D}_\pm^\ast 
(\epsilon) = \mathcal{D}_\pm^\ast (\epsilon)$.
Near $-m$, for ${\rm Arg}\, \Phi \in \pm \big( 
\frac{\pi}{2},\pi \big)$ and $\delta$
small enough, we have respectively $\pm \Phi 
\mathcal{C}_\delta(J) \cap \mathcal{D}_\mp^\ast 
(\epsilon) = \mathcal{D}_\mp^\ast (\epsilon)$.
Therefore, Corollary \ref{c1} follows according 
to \eqref{eq2,13} and \eqref{eq2,14}.
\\

\noindent
Similarly to \eqref{eq2,2}, let 
$\textbf{\textup{W}}_{\pm m}$ define the 
multiplication operators by the functions 
$\textbf{\textup{W}}_{\pm m} : \brrr^2 
\longrightarrow \brrr$ with respect to 
the matrix $\vert W \vert$. Hence, let us
consider the following
\\

\noindent
\textbf{Assumption 2.2.}
The functions $\textbf{\textup{W}}_{\pm m}$
satisfy $0 < \textbf{\textup{W}}_{\pm m}(x_1,x_2) 
\leq e^{-C \langle (x_1,x_2) \rangle^2}$
for some positive constant $C$.
\\

\noindent
For $r_{0} > 0$, $\delta > 0$ two fixed 
constants, and $r > 0$ which tends to 
zero, we define
\begin{equation}\label{eq2,26}
\Gamma^{\delta}(r,r_{0}) := \big\lbrace x + 
iy \in \mathbb{C} : r < x < r_{0}, -\delta x 
< y < \delta x \big\rbrace.
\end{equation}

\begin{theo}[\textbf{Lower bounds}]\label{t2,4}
Assume that $V$ satisfies Assumptions 1.1, 
2.1 and 2.2 with $W \geq 0$. Then, for any 
$\delta > 0$ small enough, there exists 
$\varepsilon_0 > 0$ such that for any 
$0 < \varepsilon \leq \varepsilon_0$, there is 
an accumulation of non-real eigenvalues 
$z_{\pm m}(k)$ of $D_m(b,\varepsilon V)$ near 
$\pm m$ in a sector around the semi-axis 
$\footnote{For $r \in \brrr$, we set $r_\pm := 
\max(0,\pm r)$.}$
\begin{equation}
\begin{cases} 
 z = \pm m \pm e^{i(2{\rm Arg}\, \Phi - \pi)} 
 ]0,+\infty) & \text{for } {\rm Arg}\, \Phi 
 \in \left( \frac{\pi}{2} \right)_\pm + \left( 0,
 \frac{\pi}{2}\right), \\
 z = \pm m \pm e^{i(2{\rm Arg}\, \Phi + \pi)} 
 ]0,+\infty) & \text{for } {\rm Arg}\, \Phi 
 \in -\left( 
 \frac{\pi}{2} \right)_\pm - \left( 0,
 \frac{\pi}{2}\right).
 \end{cases}
\end{equation} 
More precisely, for 
\begin{equation}
{\rm Arg}\, \Phi \in \left( \frac{\pi}{2} 
\right)_\pm + \left( 0,\frac{\pi}{2} \right),
\end{equation}
there exists a decreasing 
sequence of positive numbers $(r_\ell^{\pm m})$, 
$r_\ell^{\pm m} \searrow 0$, such that
\begin{equation}\label{eq2,30}
\displaystyle 
\sum_{\substack{z_{\pm m}(k) \hspace{0.5mm} 
\in \hspace{0.5mm} 
{\rm sp}\,_{\textup{\textbf{disc}}}^+ 
\big( D_m(b,\varepsilon V) \big) \\ k \hspace{0.5mm} 
\in \hspace{0.5mm} -iJ \Phi
\Gamma^\delta(r_{\ell +1}^{\pm m},r_\ell^{\pm m}) 
\cap \mathcal{D}_\pm^\ast(\epsilon)}} \textup{mult} 
\big( z_{\pm m}(k) \big) \geq \textup{Tr} 
\hspace{0.4mm} 
\one_{(r_{\ell +1}^{\pm m},r_\ell^{\pm m})} 
\big( p \textbf{\textup{W}}_{\pm m} p \big).
\end{equation}
For 
\begin{equation}
{\rm Arg}\, \Phi \in - \left( \frac{\pi}{2} 
\right)_\pm - \left( 0,\frac{\pi}{2} \right),
\end{equation}
\eqref{eq2,30} holds again with  
${\rm sp}\,_{\textup{\textbf{disc}}}^+ \big( 
D_m(b,\varepsilon V) \big)$ replaced by 
${\rm sp}\,_{\textup{\textbf{disc}}}^- \big( 
D_m(b,\varepsilon V) \big)$, $k$ by $-k$, and 
$\mathcal{D}_\pm^\ast(\epsilon)$ by 
$\mathcal{D}_\mp^\ast(\epsilon)$.
\end{theo}

\noindent
A graphic illustration of 
Theorems \ref{t2,3} and \ref{t2,4} near 
$m$ with $V = \Phi W$, $W \geq 0$, is 
given in Figure 2.1.

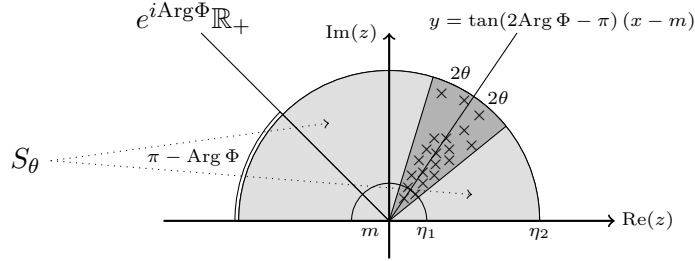
\begin{figure}[h]\label{fig 2}
\begin{center}
\tikzstyle{+grisEncadre}=[fill=gray!60]
\tikzstyle{blancEncadre}=[fill=white!100]
\tikzstyle{grisEncadre}=[fill=gray!25]
\tikzstyle{dEncadre}=[dotted]

\begin{tikzpicture}[scale=1]

\draw (0,0) -- (0:2) arc (0:180:2) -- cycle;

\node at (-0.25,-0.15) {\tiny{$m$}};
\node at (0.5,-0.16) {\tiny{$\eta_1$}};
\node at (2,-0.17) {\tiny{$\eta_2$}};

\draw [grisEncadre] (0,0) -- (0:2) arc (0:180:2) -- cycle;

%\draw [blancEncadre] (0,0) -- (0:0.5) arc (0:90:0.5) -- cycle;

\draw [+grisEncadre] (0,0) -- (39:2) arc (39:73:2) -- cycle;

\draw (0,0) -- (0:0.5) arc (0:180:0.5) -- cycle;

\draw (0,0) -- (-2.5,2.5) -- cycle;
\draw (-2.6,2.35) node[above] {$e^{i{\rm Arg} \Phi} \mathbb{R}_+$};
\draw (-2.6,0.6) node[above] {\tiny{$\pi - {\rm Arg}\, \Phi$}};
\draw (0,0) -- (135:2.05) arc (135:180:2.05) -- cycle;

\draw[->] [thick] (-3,0) -- (3,0);
\draw (2.96,0) node[right] {\tiny{$\re(z)$}};

\draw[->] [thick] (0,-0.5) -- (0,2.5);
\node at (-0.5,2.5) {\tiny{$\im(z)$}};

\draw (0,0) -- (1.7,2.5);
\draw (2.3,2.4) node[above] {\tiny{$y = \tan (2{\rm Arg}\, \Phi - \pi) 
\hspace{0.5mm} (x - m)$}};

\draw (0.95,1.75) node[above] {\tiny{$2\theta$}};
\draw (1.45,1.42) node[above] {\tiny{$2\theta$}};

\draw [dEncadre] [->] (-4.5,0.8) -- (-0.8,1.3);
\draw [dEncadre] [->] (-4.5,0.8) -- (1.1,0.35);
\draw (-4.5,0.8) node[left] {$S_\theta$};

\node at (0.45,0.5) {\tiny{$\times$}};
\node at (0.45,0.65) {\tiny{$\times$}};
\node at (0.63,0.9) {\tiny{$\times$}};
\node at (0.76,1.1) {\tiny{$\times$}};
\node at (0.8,0.95) {\tiny{$\times$}};
\node at (0.62,0.75) {\tiny{$\times$}};

\node at (0.3,0.6) {\tiny{$\times$}};
\node at (0.4,0.8) {\tiny{$\times$}};
\node at (0.6,1.1) {\tiny{$\times$}};
\node at (0.5,0.95) {\tiny{$\times$}};

\node at (0.35,0.35) {\tiny{$\times$}};
\node at (0.2,0.3) {\tiny{$\times$}};
\node at (0.25,0.45) {\tiny{$\times$}};
\node at (0.6,0.6) {\tiny{$\times$}};
\node at (0.8,0.8) {\tiny{$\times$}};
\node at (1.2,1.4) {\tiny{$\times$}};
\node at (1,1.6) {\tiny{$\times$}};
\node at (0.7,1.7) {\tiny{$\times$}};
\node at (1,1.2) {\tiny{$\times$}};
\node at (1.1,1) {\tiny{$\times$}};

\node at (0,4) {$V = \Phi W$};
\node at (0,4.48) {${\rm Arg}\, \Phi \in (\frac{\pi}{2},\pi), \hspace{0.5mm} W \geq 0$};

\end{tikzpicture}
\caption{\textbf{Localisation of the non-real
eigenvalues near $m$ with $0 < \eta_1 < \eta_2 
< \eta$ small enough:} For $\theta$ small 
enough and $0 < \varepsilon \leq \varepsilon_0$,
$D_m(b,\varepsilon V) := D_m(b,0) + \varepsilon V$ 
has no eigenvalues in $S_\theta$ 
(Theorem \ref{t2,3}). They are concentrated 
around the semi-axis $z = m + e^{i(2{\rm Arg}\, 
\Phi - \pi)} ]0,+\infty)$ (Theorem \ref{t2,4}).}
\end{center}
\end{figure}

\begin{figure}[h]\label{fig 1}
\begin{center}

\vspace*{-1.5cm}

\hspace*{-3.3cm} \includegraphics[scale=0.7]{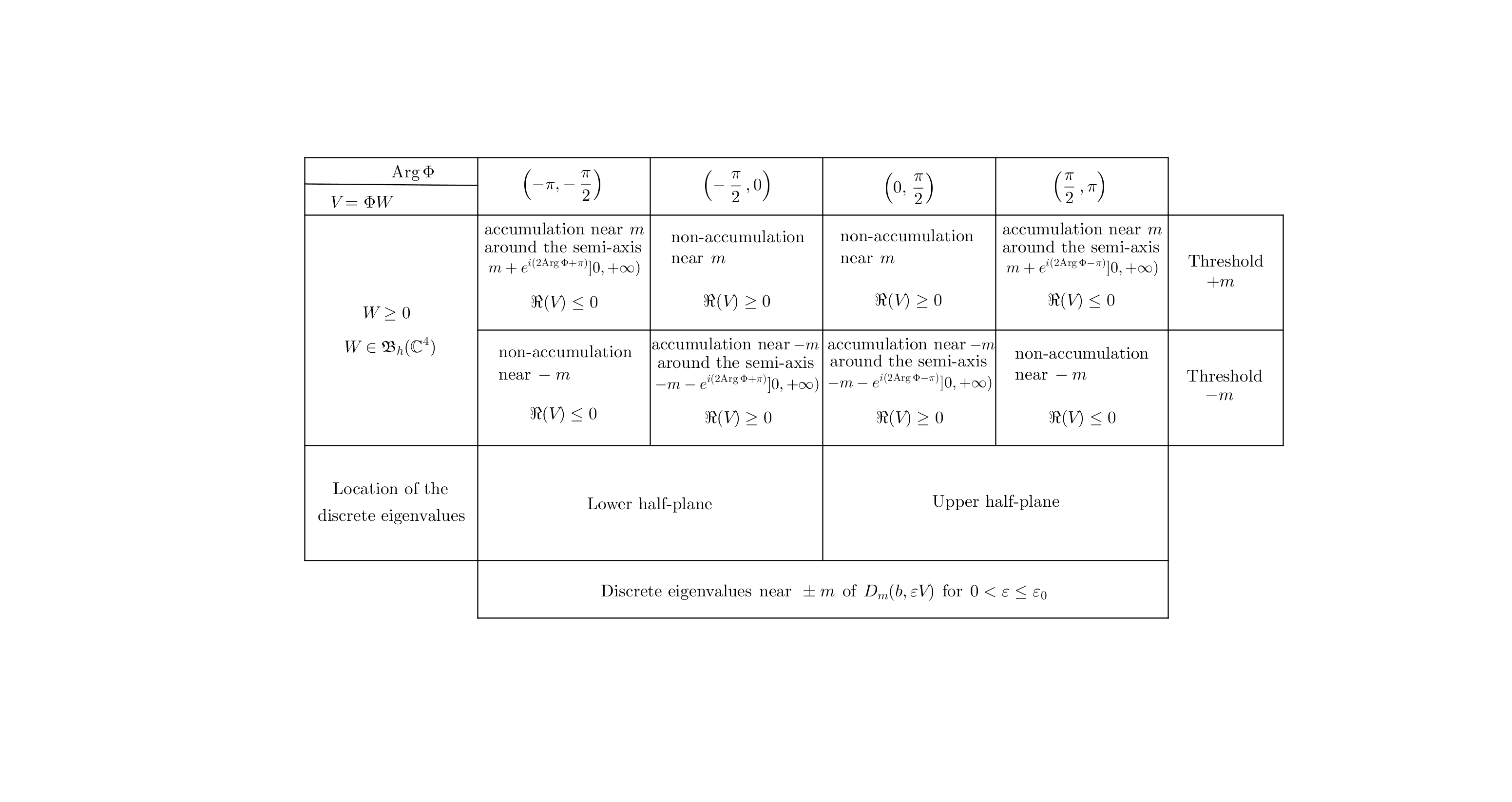}

\vspace*{-1.5cm}

\caption{Summary of results.}
\end{center}
\end{figure}

\medskip

\noindent
In the above results, the accumulation of 
the non-real eigenvalues of $D_m(b,\varepsilon 
V)$ near $\pm m$ holds for any $0 < \varepsilon 
\leq \varepsilon_0$. We expect this to be 
a general phenomenon in the sense of the 
following conjecture:

\begin{cnj}\label{cj2,1}
Let $V = \Phi W$ satisfy Assumption 1.1 
with ${\rm Arg}\, \Phi \in \bc \setminus 
\brrr e^{ik \frac{\pi}{2}}$, $k \in \bz$, 
and $W$ Hermitian of definite sign. 
Then, for any domain $\Omega$ as in 
$\eqref{defNq}$, we have
\begin{equation}
\mathcal{N}_{\pm m} \big( D_m(b,V),\Omega 
\big) < \infty
\end{equation} 
if and only if
%\begin{equation} 
$\pm \re (V) > 0$.
%\end{equation} 
\end{cnj}

\section{Characterisation of the discrete eigenvalues}\label{s3}

From now on, for $\widetilde m \in \lbrace 
\pm m \rbrace$, $\mathcal{D}^\pm_{\widetilde m} 
(\eta)$ and $\mathcal{D}_\pm^\ast (\epsilon)$ 
are the domains given by \eqref{eq2,8} and 
\eqref{eq2,10} respectively.

\subsection{Local properties of the (weighted) free resolvent}\label{s4,1}

In this subsection, we show in particular 
that under Assumption $1.1$, $V$ is relatively 
compact with respect to $D_m(b,0)$.
\\

\noindent
Let $P := p \otimes 1$ define the orthogonal 
projection onto $\text{Ker} \hspace{0.5mm} H_\perp^- 
\otimes L^2(\brrr)$, where $H_\perp^-$ is the 
two-dimensional magnetic Schrödinger operator 
defined by \eqref{eqh}. Denote $\textup{\textbf{P}}$ 
the orthogonal projection onto the union of 
the eigenspaces of $D_m(b,0)$ corresponding 
to $\pm m$. Then, we have
\begin{equation}\label{eq4,1}
\textup{\textbf{P}} = 
\left( \begin{smallmatrix}
   P & 0 & 0 & 0 \\
   0 & 0 & 0 & 0 \\
   0 & 0 & P & 0 \\
   0 & 0 & 0 & 0
\end{smallmatrix} \right) \quad \text{and} \quad 
\textup{\textbf{Q}} := 
\textup{I} - \textup{\textbf{P}} = 
\left( \begin{smallmatrix}
   I - P & 0 & 0 & 0 \\
   0 & I & 0 & 0 \\
   0 & 0 & I - P & 0 \\
   0 & 0 & 0 & I
\end{smallmatrix} \right),
\end{equation}
(see \cite[Section 3]{diom}). Moreover, if 
$z \in \bc \setminus (-\infty,-m] \cup 
[m,+\infty)$, then 
\begin{equation}\label{eq4,2}
\big( D_m(b,0) - z \big)^{-1} = \big( D_m(b,0) - 
z \big)^{-1} \textup{\textbf{P}} + \big( D_m(b,0) 
- z \big)^{-1} \textup{\textbf{Q}}
\end{equation}
with 
\begin{equation}\label{eq4,4}
\begin{split}
\big( D_m(b,0) - z \big)^{-1} \textup{\textbf{P}}
 = \Big[ p \otimes & \mathscr{R}(z^2 - m^2) \Big] 
 \left( \begin{smallmatrix}
   z + m & 0 & 0 & 0 \\
   0 & 0 & 0 & 0 \\
   0 & 0 & z - m & 0 \\
   0 & 0 & 0 & 0
\end{smallmatrix} \right) \\
& + \Big[ p \otimes (-i\partial_{x_3}) 
\mathscr{R}(z^{2} - m^{2}) \Big] 
\left( \begin{smallmatrix}
   0 & 0 & 1 & 0 \\
   0 & 0 & 0 & 0 \\
   1 & 0 & 0 & 0 \\
   0 & 0 & 0 & 0
\end{smallmatrix} \right).
\end{split}
\end{equation}
Here, the resolvent $\mathscr{R}(z) := \left( 
-\partial_{x_3}^2 - z \right)^{-1}$, $z \in 
\bc \setminus [0,+\infty)$, acts in $L^{2}(\brrr)$.
It admits the integral kernel
\begin{equation}\label{eq4,5}
I_z(x_3,x_3') := -\frac{e^{i\sqrt{z} \vert x_3 - 
x_3' \vert}}{2i\sqrt{z}},
\end{equation}  
%according to our choice of the square root 
%\eqref{eq2,7}
with $\im \big( \sqrt{z} \big) > 0$. In what 
follows below, the definition of the 
Schatten-von Neumann class ideals $\sqq$ is 
recalled in Appendix A.

\begin{lem}\label{l4,1}
Let $U \in L^q(\brrr^2)$, $q \in [2,+\infty)$ 
and $\tau > \frac{1}{2}$. Then, the 
operator-valued function 
$$
\bc \setminus {\rm sp}\, \big( D_m(b,0) \big) \ni 
z \longmapsto U \langle x_3 \rangle^{-\tau} 
\big( D_m(b,0) - z \big)^{-1} \textup{\textbf{P}}
$$ 
is holomorphic with values in $\sqq 
\big( L^{2}(\brrr^3) \big)$. Moreover, we have
\begin{equation}\label{eq4,6}
\small{\Big\Vert U \langle x_3 \rangle^{-\tau}
\big( D_m(b,0) - z \big)^{-1} \textup{\textbf{P}} 
\Big\Vert_\sqq^q \leq
C \Vert U \Vert_{L^q}^q M(z,m)^q,}
\end{equation}
where
\begin{equation}\label{eq4,61}
\begin{split}
\small{M(z,m) := \Vert \langle x_3 \rangle^{-\tau} 
\Vert_{L^q} \hspace{0.5mm} \big(} & \small{\vert z 
+ m \vert + \vert z - m \vert \big) \textup{sup}_{s 
\in [0,+\infty)} \left\vert \frac{s + 1}{s - z^2 + 
m^2} \right\vert} \\
& \small{+ \frac{\Vert \langle x_3 \rangle^{-\tau} 
\Vert_{L^2}}{\big( \im\sqrt{z^2 - m^2} 
\big)^{\frac{1}{2}}},}
\end{split}
\end{equation}
$C = C(q,b)$ being a constant depending on 
$q$ and $b$.
\end{lem}

\begin{proof}
The holomorphicity on $\bc \setminus {\rm sp}\, 
\big( D_m(b,0) \big)$ is evident. Let us 
prove the bound \eqref{eq4,6}. 
Constants are generic (i.e. changing from 
a relation to another).
Set 
\begin{equation}\label{eq4,7}
L_1(z) := \Big[ p \otimes \mathscr{R}(z^2 - 
m^2) \Big]  \left( \begin{smallmatrix}
   z + m & 0 & 0 & 0 \\
   0 & 0 & 0 & 0 \\
   0 & 0 & z - m & 0 \\
   0 & 0 & 0 & 0
\end{smallmatrix} \right)
\end{equation}
and 
\begin{equation}\label{eq4,8}
L_2(z) := \Big[ p \otimes (-i\partial_{x_3}) 
\mathscr{R}(z^{2} - m^{2}) \Big] 
\left( \begin{smallmatrix}
   0 & 0 & 1 & 0 \\
   0 & 0 & 0 & 0 \\
   1 & 0 & 0 & 0 \\
   0 & 0 & 0 & 0
\end{smallmatrix} \right).
\end{equation}
Then, from \eqref{eq4,4}, we get
\begin{equation}\label{eq4,9}
U \langle x_3 \rangle^{-\tau} \big( D_m(b,0) 
- z \big)^{-1} \textup{\textbf{P}} = 
U \langle x_3 \rangle^{-\tau} L_1(z) + 
U \langle x_3 \rangle^{-\tau} L_2(z).
\end{equation}

First, we estimate the $\sqq$-norm of the
first term of the RHS of \eqref{eq4,9}. 
Thanks to \eqref{eq4,7}, we have
\begin{equation}\label{eq4,10}
U \langle x_3 \rangle^{-\tau} L_1(z) = 
\Big[ U p \otimes \langle x_3 \rangle^{-\tau} 
\mathscr{R}(z^2 - m^2) \Big]  
\left( \begin{smallmatrix}
   z + m & 0 & 0 & 0 \\
   0 & 0 & 0 & 0 \\
   0 & 0 & z - m & 0 \\
   0 & 0 & 0 & 0
\end{smallmatrix} \right).
\end{equation}  
By an easy adaptation of 
\cite[Proof of Lemma 2.4]{raik}, it can be 
similarly proved that the operator $U p$
satisfies $U p \in \sqq \big( L^{2}(\brrr^2) 
\big)$ with
\begin{equation}\label{eq4,11}
\begin{split}
& \Vert U p \Vert_\sqq^q \leq
\frac{b_{0}}{2\pi} \textup{e}^{2 \hspace{0.2mm} \textup{osc} \hspace{0.5mm} \tilde{\varphi}} \Vert U \Vert_{L^q}^q, \\
& \textup{osc} \hspace{0.5mm} \tilde{\varphi}:= \displaystyle\sup_{(x_1,x_2) \in \mathbb{R}^{2}} \tilde{\varphi} (x_1,x_2) \, - \displaystyle\inf_{(x_1,x_2) \in \mathbb{R}^{2}} \tilde{\varphi} (x_1,x_2).
\end{split}
\end{equation}
On the other hand, we have
\begin{equation}\label{eq4,12}
\begin{split}
\left\Vert \langle x_3 \rangle^{-\tau} 
\mathscr{R}(z^2 - m^2) \right\Vert_\sqq^q 
\leq & \left\Vert \langle x_3 \rangle^{-\tau} 
\left( -\partial_{x_3}^2 + 1 \right)^{-1} 
\right\Vert_\sqq^q \\ 
& \times \left\Vert \left( -\partial_{x_3}^2 + 
1 \right) \mathscr{R}(z^2 - m^2) \right\Vert^q.
\end{split}
\end{equation}
By the Spectral mapping theorem, we have
\begin{equation}\label{eq4,13}
\left\Vert \left( -\partial_{x_3}^2 + 1 \right)
\mathscr{R}(z^2 - m^2) \right\Vert^q \leq 
\textup{sup}_{s \in [0,+\infty)}^q 
\left\vert \frac{s + 1}{s - z^2 + m^2} 
\right\vert,
\end{equation}
and by the standard criterion
\cite[Theorem 4.1]{sim}, we have
\begin{equation}\label{eq4,14}
\left\Vert \langle x_3 \rangle^{-\tau} \left( 
-\partial_{x_3}^2 + 1 \right) \right\Vert^q_\sqq 
\leq C \Vert \langle x_3 \rangle^{-\tau} 
\Vert_{L^q}^q \left\Vert \Bigl( \vert \cdot 
\vert^{2} + 1 \Bigr)^{-1} \right\Vert_{L^q}^q.
\end{equation}
By combining \eqref{eq4,10}, \eqref{eq4,11}, 
\eqref{eq4,12}, \eqref{eq4,13} with \eqref{eq4,14},
we get
\begin{equation}\label{eq4,15}
\begin{split}
\big\Vert U \langle x_3 \rangle^{-\tau} & L_1(z) 
\big\Vert_\sqq^q \leq C(q,b) \Vert U 
\Vert_{L^q}^q \Vert \langle x_3 \rangle^{-\tau} 
\Vert_{L^q}^q \\
& \times \big( \vert z + m \vert + \vert z - m 
\vert \big)^q \textup{sup}_{s \in [0,+\infty)}^q 
\left\vert \frac{s + 1}{s - z^2 + m^2} \right\vert.
\end{split}
\end{equation}

Now, we estimate the $\sqq$-norm of the
second term of the RHS of \eqref{eq4,9}.
Thanks to \eqref{eq4,8}, we have
\begin{equation}\label{eq4,16}
U \langle x_3 \rangle^{-\tau} L_2(z) = 
\Big[ U p \otimes \langle x_3 \rangle^{-\tau} 
(-i\partial_{x_3}) \mathscr{R}(z^2 - m^2) \Big]  
\left( \begin{smallmatrix}
   0 & 0 & 1 & 0 \\
   0 & 0 & 0 & 0 \\
   1 & 0 & 0 & 0 \\
   0 & 0 & 0 & 0
\end{smallmatrix} \right).
\end{equation}  
According to \eqref{eq4,5}, the operator 
$\langle x_3 \rangle^{-\tau} (-i\partial_{x_3}) 
\mathscr{R}(z^2 - m^2)$ admits the integral 
kernel
\begin{equation}\label{eq4,17}
- \langle x_3 \rangle^{-\tau} \frac{x_3 - x_3'}
{2\vert x_3 - x_3' \vert} e^{i\sqrt{z^2 - m^2} 
\vert x_3 - x_3' \vert}.
\end{equation} 
An estimate of the $L^2(\brrr^2)$-norm of 
\eqref{eq4,17} shows that $\langle x_3 
\rangle^{-\tau} (-i\partial_{x_3}) \mathscr{R}
(z^2 - m^2) \in \sd \big( L^2(\brrr) \big)$
with
\begin{equation}\label{eq4,18}
\big\Vert \langle x_3 \rangle^{-\tau} 
(-i\partial_{x_3}) \mathscr{R}
(z^2 - m^2) \big\Vert_\sd^2 \leq 
\frac{C \Vert \langle x_3 \rangle^{-\tau} 
\Vert_{L^2}^2}{\im\sqrt{z^2 - m^2}}.
\end{equation}  
By combining \eqref{eq4,16}, \eqref{eq4,11} 
with \eqref{eq4,18}, we get
\begin{equation}\label{eq4,19}
\big\Vert U \langle x_3 \rangle^{-\tau} 
L_2(z) \big\Vert_\sqq \leq C(q,b)^{\frac{1}{q}} 
\frac{\Vert U \Vert_{L^q} \Vert 
\langle x_3 \rangle^{-\tau} \Vert_{L^2}}{\big( 
\im\sqrt{z^2 - m^2} \big)^{\frac{1}{2}}}.
\end{equation}
Then, \eqref{eq4,6} follows immediately 
from \eqref{eq4,9}, \eqref{eq4,15} and 
\eqref{eq4,19}, which gives the proof.
\end{proof}

\noindent
For simplicity of notation in the sequel, 
we set
\begin{equation}\label{eq4,22}
H^\pm := (-i\nabla - \textbf{A})^{2} \pm b
= H_\perp^\pm \otimes 1 + 1 \otimes
(-\partial_{x_3}^2),
\end{equation}
where $H_\perp^-$ is the operator 
defined by \eqref{eqh}, $H_\perp^+$ being the 
corresponding operator with $-b$ replaced 
by $b$. 
%Here, $I_3$ and $I_\perp$ define 
%respectively the identity operators in
%$L^2(\brrr)$ and $L^2(\brrr^2)$. 
We recall
from \cite[Subsection 2.2]{raik} that 
we have 
\begin{equation}\label{eqs}
\begin{split}
& \dim \hspace{0.5mm} \textup{Ker} 
\hspace{0.5mm} H_{\perp}^{-} = \infty, 
\quad \dim \hspace{0.5mm} \textup{Ker} 
\hspace{0.5mm} H_{\perp}^{+} = 0, \\
& \text{and} \quad \sigma (H_{\perp}^{\pm}) 
\subset \lbrace 0 \rbrace \cup [\zeta,+\infty), 
\end{split}
\end{equation}
with
\begin{equation}
\zeta := 2 b_{0} \textup{e}^{-2 \textup{osc} 
\hspace{0.5mm} \tilde{\varphi}} > 0,
\end{equation}
$\textup{osc} 
\hspace{0.5mm} \tilde{\varphi}$ being defined 
by \eqref{eq4,11}. Since the spectrum of
the one-dimensional Laplacian $-\partial_{x_3}^2$ 
coincides with $[0,+\infty)$, we deduce
from \eqref{eq4,22} and \eqref{eqs} that, on 
one hand, the spectrum of the operator $H^+$ 
belongs to $[\zeta,+\infty)$ (notice that in the 
constant magnetic field case $b = b_0$, we have 
$\zeta = 2b_0$, the first Landau level of $H^+$). 
On the other hand, that the spectrum of the
operator $H^-$ coincides with $[0,+\infty)$.

\begin{lem}\label{l4,2} 
Let $g \in L^q(\brrr^3)$, $q \in [4,+\infty)$. 
Then, the operator-valued function 
\begin{equation}\label{eq4,191}
\small{\bc \setminus \left\lbrace \left( -\infty,
-\sqrt{m^{2} + \zeta} \right] \cup \left[ 
\sqrt{m^2 + \zeta},+\infty \right) \right\rbrace 
\ni z \longmapsto g \big( D_m(b,0) 
- z \big)^{-1} \textup{\textbf{Q}}}
\end{equation}
is holomorphic with values in $\sqq 
\big( L^{2}(\brrr^3) \big)$. Moreover, we have
\begin{equation}\label{eq4,20}
\left\Vert g \big( D_m(b,0) - z \big)^{-1} 
\textup{\textbf{Q}} \right\Vert_\sqq^q \leq C \Vert 
g \Vert_{L^{q}}^q \widetilde{M}(z,m)^q,
\end{equation}
where
\begin{equation}\label{eq4,21}
\small{\widetilde{M}(z,m) :=} \small{\textup{sup}_{s \in [\zeta,+\infty)} 
\left\vert \frac{s + 1}{s + m^2} \right\vert^{\frac{1}{2}}} 
 \small{+ \hspace*{0.1cm} \big( \vert z \vert + \vert z \vert^2 \big) 
\textup{sup}_{s \in [\zeta,+\infty)} 
\left\vert \frac{s + 1}{s + m^2 - z^2} \right\vert,}
\end{equation}
$C = C(q)$ being a constant depending on $q$.
\end{lem}

\begin{proof}
For $z \in \rho \big( D_m(b,0) \big)$ \big(the 
resolvent set of $D_m(b,0)$\big), we have
\begin{equation}\label{eq4,23}
\small{\big( D_m(b,0) - z \big)^{-1} = D_m(b,0)^{-1} + z 
\big( 1 + z D_m(b,0)^{-1} \big) \big( D_m(b,0)^{2} 
- z^{2} \big)^{-1}.}
\end{equation} 
By setting 
\begin{equation}\label{eq4,24}
L_3(z) := z 
\big( 1 + z D_m(b,0)^{-1} \big) \big( D_m(b,0)^{2} 
- z^{2} \big)^{-1},
\end{equation} 
we get from \eqref{eq4,23}
\begin{equation}\label{eq4,25}
g \big( D_m(b,0) - z \big)^{-1} \textup{\textbf{Q}} 
= g D_m(b,0)^{-1}\textup{\textbf{Q}} + g L_3(z) 
\textup{\textbf{Q}}.
\end{equation} 
It can be proved that
\begin{equation}\label{eq4,26}
\begin{split}
& \tiny{\big( D_m(b,0)^{2} -z^2 \big)^{-1} 
\textup{\textbf{Q}}}
\\ 
& = \tiny{\left( \begin{smallmatrix}
 \big( H^- + m^2 -z^2 \big)^{-1}(I - P) 
 & 0 & 0 & 0 \\
 0 & \big( H^+ + m^2 - z^2 \big)^{-1} & 0 & 0 \\
 0 & 0 & \big( H^- + m^2 -z^2 \big)^{-1}(I - P) 
 & 0 \\
 0 & 0 & 0 & \big( H^+ + m^2 - z^2 \big)^{-1}
\end{smallmatrix} \right)},
\end{split}
\end{equation}
(see for instance \cite[Identity (2.2)]{tda}). 
The set $\bc \setminus [\zeta,+\infty)$ is 
included in the resolvent set of $H^-$ 
defined on $(I - P)Dom(H^-)$. Similarly, it is
included in the resolvent set of $H^+$
defined on $Dom(H^+)$. Then,
\begin{equation}\label{eq4,27}
\small{\bc \setminus \left\lbrace \left( 
-\infty,-\sqrt{m^{2} + \zeta} \right] \cup 
\left[ \sqrt{m^2 + \zeta},+\infty \right) 
\right\rbrace \ni z \longmapsto \big( D_m(b,0)^2 
- z^2 \big)^{-1} \textup{\textbf{Q}}}
\end{equation}
is well defined and holomorphic. Therefore, so 
is the operator-valued function \eqref{eq4,191} 
thanks to \eqref{eq4,24} and \eqref{eq4,25}.

It remains to prove the bound \eqref{eq4,20}. 
As in the proof of the previous lemma, the 
constants change from a relation to another. 
First, we prove that \eqref{eq4,20} is true 
for $q$ even. 

Let us focus on the second term of the RHS of 
\eqref{eq4,25}. According to \eqref{eq4,24} and
\eqref{eq4,26}, we have
\begin{equation}\label{eq4,28}
\begin{split}
& \left\Vert g L_3(z) \textup{\textbf{Q}}
\right\Vert_\sqq^q \leq C \big( \vert z \vert + 
\vert z \vert^2 \big)^q \\
& \times \left( \left\Vert g \big( H^- + m^2 - z^2 
\big)^{-1}(I - P) \right\Vert_\sqq^q + \left\Vert g 
\big( H^+ + m^2 - z^2 \big)^{-1} \right\Vert_\sqq^q 
\right).
\end{split}
\end{equation}
One has
\begin{equation}\label{eq4,29}
\begin{split}
\Big\Vert g \big( H^- + m^2 - z^2 \big)^{-1} & (I - P) 
\Big\Vert_\sqq^q \leq \left\Vert g (H^- + 1)^{-1} 
\right\Vert_\sqq^q \\
& \times \left\Vert (H^- + 1) \big( H^- + m^2 - z^2 
\big)^{-1}(I - P) \right\Vert^q.
\end{split}
\end{equation}
The Spectral mapping theorem implies that
\begin{equation}\label{eq4,30}
\small{\left\Vert (H^- + 1)\big( H^- + m^2 - z^2 \big)^{-1}
(I - P) \right\Vert^q \leq \textup{sup}_{s \in 
[\zeta,+\infty)}^q \left\vert \frac{s + 1}{s + m^2 - 
z^2} \right\vert.}
\end{equation}
Exploiting the resolvent equation, the 
boundedness of $b$, and the diamagnetic inequality 
(see \cite[Theorem 2.3]{avr} and \cite[Theorem 
2.13]{sim}, which is only valid when $q$ is even), 
we obtain
\begin{equation}\label{eq4,31}
\begin{split}
\left\Vert g \bigl( H^- + 1 \bigr)^{-1} 
\right\Vert^q_\sqq & \leq \left\Vert I + 
(H^- + 1)^{-1}b \right\Vert^q \left\Vert g \big( 
(-i\nabla - \textbf{A})^{2} + 1 \big)^{-1} 
\right\Vert^q_\sqq \\ & \leq C \left\Vert g 
(-\Delta + 1)^{-1} \right\Vert^q_\sqq.
\end{split}
\end{equation}
The standard criterion \cite[Theorem 4.1]{sim} 
implies that
\begin{equation}\label{eq4,32}
\left\Vert g (-\Delta + 1 \vert)^{-1} 
\right\Vert^q_\sqq \leq C \Vert g \Vert_{L^q}^q 
\left\Vert \Bigl( \vert \cdot \vert^{2} + 1 
\Bigr)^{-1} \right\Vert_{L^q}^q.
\end{equation}
The bound \eqref{eq4,29} together with \eqref{eq4,30}, 
\eqref{eq4,31} and \eqref{eq4,32} give
\begin{equation}\label{eq4,33}
\small{\left\Vert g \big( H^- + m^2 - z^2 \big)^{-1}
(I - P) \right\Vert^q_\sqq
\leq C \Vert g \Vert_{L^{q}}^{q} \textup{sup}_{s 
\in [\zeta,+\infty)}^q \left\vert \frac{s + 1}
{s + m^2 - z^2} \right\vert.}
\end{equation}
Similarly, it can be shown that
\begin{equation}\label{eq4,34}
\left\Vert g \big( H^+ + m^2 - z^2 \big)^{-1} 
\right\Vert^q_\sqq \leq C \Vert g 
\Vert_{L^{q}}^{q} \textup{sup}_{s \in 
[\zeta,+\infty)}^q \left\vert \frac{s + 1}
{s + m^2 - z^2} \right\vert.
\end{equation}
This together with \eqref{eq4,28} and 
\eqref{eq4,33} give
\begin{equation}\label{eq4,35}
\left\Vert g L_3(z) \textup{\textbf{Q}}
\right\Vert_\sqq^q \leq C \Vert g 
\Vert_{L^{q}}^{q} \big( \vert z \vert + 
\vert z \vert^2 \big)^q \textup{sup}_{s \in 
[\zeta,+\infty)}^q \left\vert \frac{s + 1}
{s + m^2 - z^2} \right\vert.
\end{equation}

Now, we focus on the first term $g D_m(b,0)^{-1} 
\textup{\textbf{Q}}$ of the RHS of \eqref{eq4,25}. 
For $\gamma > 0$, as in \eqref{eq4,26}, we have
\begin{equation}\label{eq4,36}
\begin{split}
& \tiny{D_m(b,0)^{-\gamma} 
\textup{\textbf{Q}}} \\ 
& = \tiny{\left( \begin{smallmatrix}
 \big( H^- + m^2 \big)^{-\frac{\gamma}{2}}(I - P) 
 & 0 & 0 & 0 \\
 0 & \big( H^+ + m^2 \big)^{-\frac{\gamma}{2}} 
 & 0 & 0 \\
 0 & 0 & \big( H^- + m^2 \big)^{-\frac{\gamma}{2}}
 (I - P)  
 & 0 \\
 0 & 0 & 0 & \big( H^+ + m^2 
 \big)^{-\frac{\gamma}{2}}
\end{smallmatrix} \right).}
\end{split}
\end{equation}
Therefore, arguing as above 
\big(\eqref{eq4,28}-\eqref{eq4,34}\big), it 
can be proved that 
\begin{equation}\label{eq4,37}
\left\Vert g D_m(b,0)^{-\gamma} \textup{\textbf{Q}}
\right\Vert_\sqq^q \leq C(q,\gamma) \Vert g 
\Vert_{L^{q}}^{q} \textup{sup}_{s \in 
[\zeta,+\infty)}^q \left\vert \frac{s + 1}
{s + m^2} \right\vert^\frac{\gamma}{2}, \quad
\gamma q > 3.
\end{equation}
Then, for $q$ even, \eqref{eq4,20} follows 
by putting together \eqref{eq4,25}, 
\eqref{eq4,35}, and \eqref{eq4,37} with 
$\gamma = 1$.

We get the general case $q \geq 4$ with 
the help of interpolation methods. 

If $q$ satisfies $q > 4$, then, there exists 
even integers $q_{0} < q_{1}$ such that $q 
\in (q_{0},q_{1})$ with $q_{0} \geq 4$. 
Let $\beta \in (0,1)$ satisfy $\frac{1}{q} 
= \frac{1 - \beta}{q_{0}} + \frac{\beta}{q_{1}}$ 
and consider the operator
$$
L^{q_i} \big( \brrr^3 \big) \ni g \overset{T}
{\longmapsto} g \big( D_m(b,0) - z \big)^{-1} 
\textup{\textbf{Q}} \in \textup{\textbf{S}}_{q_{i}} 
\big( L^{2}(\brrr^3) \big), \qquad i = 0, 1.
$$
Let $C_{i} = C(q_{i})$, $i = 0$, $1$, 
denote the constant appearing in \eqref{eq4,20} 
and set
$$
C(z,q_{i}) := C_i^{\frac{1}{q_i}} 
\widetilde{M}(z,m).
$$
From \eqref{eq4,20}, we know that $\Vert T \Vert 
\leq C(z,q_{i})$, $i = 0$, $1$. Now, we use 
the Riesz-Thorin Theorem (see for instance 
\cite[Sub. 5 of Chap. 6]{fol}, \cite{rie,tho}, 
\cite[Chap. 2]{lun}) to interpolate between 
$q_{0}$ and $q_{1}$. We obtain the extension 
$T : L^{q}(\brrr^2) \longrightarrow \sqq \big( 
L^{2}(\brrr^3) \big)$ with
$$
\Vert T \Vert \leq C(z,q_{0})^{1-\beta} 
C(\gamma,q_{1})^{\beta} \leq C(q)^{\frac{1}{q}} 
\widetilde{M}(z,m).
$$
In particular, for any $g \in L^q(\brrr^3)$, 
we have
$$
\Vert T(g) \Vert_\sqq \leq 
C(q)^{\frac{1}{q}} \widetilde{M}(z,m)
\Vert g \Vert_{L^q},
$$
which is equivalent to \eqref{eq4,20}. This 
completes the proof.
\end{proof}

\noindent
Assumption $1.1$ ensures the existence of 
$\mathscr{V} \in \mathscr{L} \big( L^{2}(\brrr^3) 
\big)$ such that for any $x \in \brrr^3$,
\begin{equation}\label{eqv}
\vert V \vert^\frac{1}{2} (x)= \mathscr{V} 
F_\perp^\frac{1}{2} (x_1,x_2) G^\frac{1}{2}(x_3).
\end{equation}
Therefore, the boundedness of $V$ together 
with Lemmas \ref{l4,1}-\ref{l4,2}, 
\eqref{eq4,2}, and \eqref{eqv}, imply that $V$ is 
relatively compact with respect to $D_m(b,0)$.
\\

\noindent
Since for $k \in \mathcal{D}_\pm^\ast (\epsilon)$
we have $z_{\widetilde m}(k) = \frac{\widetilde 
m (1 + k^2)}{1 - k^2} \in \bc \setminus \big\lbrace 
(-\infty,-m] \cup [m,+\infty) \big\rbrace$, where
$\widetilde m \in \lbrace \pm m \rbrace$,
then this together with Lemmas \ref{l4,1}-\ref{l4,2}, 
\eqref{eq4,2} and \eqref{eqv} give the 
following 

\begin{lem}\label{l4,3} 
For $\widetilde m \in \lbrace \pm m \rbrace$ 
and $z_{\widetilde m}(k) = \frac{\widetilde 
m (1 + k^2)}{1 - k^2}$, the operator-valued 
functions
$$
\mathcal{D}_\pm^\ast (\epsilon) \ni k \longmapsto 
\mathcal{T}_{V} \big( z_{\widetilde m}(k) \big) := 
\Tilde{J} \vert V \vert^{\frac{1}{2}} \big( D_m(b,0) - 
z_{\widetilde m}(k) \big)^{-1} \vert V 
\vert^{\frac{1}{2}}
$$ 
are holomorphic with values in $\sqq \big( 
L^{2}(\brrr^3) \big)$, $\Tilde{J}$ being defined 
by the polar decomposition $V = \Tilde{J} \vert 
V \vert$.
\end{lem}

\subsection{Reduction of the problem}

We show how we can reduce the investigation of the 
discrete spectrum of $D_m(b,V)$ to that 
of zeros of holomorphic functions.
\\

\noindent
In the sequel, the definition of the $q$-regularized 
determinant $\textup{det}_{\lceil q \rceil}(\cdot)$ 
is recalled in Appendix A by \eqref{eq3,2}. 
As in Lemma \ref{l4,3}, the operator-valued function
$V \big( D_m(b,0) - \cdot \big)^{-1}$ is analytic 
on $\mathcal{D}^\pm_{\widetilde m} (\eta)$ with 
values in $\sqq \big( L^{2}(\brrr^3) \big)$. Hence, 
the following characterisation
\begin{equation}\label{eq4,38}
z \in {\rm sp}\,_{\textup{\textbf{disc}}} \big( 
D_m(b,V) \big) \Leftrightarrow f(z) := 
\textup{det}_{\lceil q \rceil} \left( I + V 
\big( D_m(b,0) - z \big)^{-1} \right) = 0
\end{equation}
holds; see for instance \cite[Chap. 9]{sim} 
for more details. The fact that the operator-valued 
function $V \big( D_m(b,0) - \cdot \big)$ 
is holomorphic on $\mathcal{D}^\pm_{\widetilde m} 
(\eta)$ implies that the same happens for the 
function $f(\cdot)$ by Property \textbf{d)} 
of Section \ref{s3,1}. Furthermore, the 
algebraic multiplicity of $z$ as discrete 
eigenvalue of $D_m(b,V)$ is equal to its order 
as zero of $f(\cdot)$ (this claim is a well 
known fact, see for instance 
\cite[Proof of Theorem 4.10 (v)]{hans} for an
idea of  proof).
\\

\noindent
In the next proposition, the quantity 
$Ind_{\mathscr{C}}( \cdot )$ in the RHS 
of \eqref{eq4,39} is recalled in Appendix B
by \eqref{eqa,2}.

%We are thus led to the following

\begin{prop}\label{p4,1}  
The following assertions are equivalent:

$\textup{\textbf{(i)}}$ $z_{\widetilde m}(k_0) = 
\frac{\widetilde m (1 + k_0^2)}{1 - k_0^2} \in 
\mathcal{D}^\pm_{\widetilde m} (\eta)$ is a 
discrete eigenvalue of $D_m(b,V)$,

$\textup{\textbf{(ii)}}$ 
$\textup{det}_{\lceil q \rceil} \left( I + 
\mathcal{T}_{V} \big( z_{\widetilde m}(k_0) \right) 
\big) = 0$,

$\textup{\textbf{(iii)}}$ $-1$ is an eigenvalue 
of $\mathcal{T}_{V} \big( z_{\widetilde m}(k_0) 
\big)$.
\\
Moreover,
\begin{equation}\label{eq4,39}
\textup{mult} \big( z_{\widetilde m}(k_0) \big) = 
Ind_{\mathscr{C}} \hspace{0.5mm} 
\Big( I + \mathcal{T}_{V}\big( z_{\widetilde m}
(\cdot) \big) \Big),
\end{equation}
$\mathscr{C}$ being a small contour positively 
oriented, containing $k_{0}$ as the unique point 
$k \in \mathcal{D}_\pm^\ast(\epsilon)$ verifying 
$z_{\widetilde m}(k) \in 
\mathcal{D}^\pm_{\widetilde m} (\eta)$ is a 
discrete eigenvalue of $D_m(b,V)$.
\end{prop}

\begin{proof}
The equivalence \textbf{(i)} $\Leftrightarrow$ 
\textbf{(ii)} follows obviously from 
\eqref{eq4,38} and the equality 
$$
\textup{det}_{\lceil q \rceil} \left( I + V 
\big( D_m(b,0) - z \big)^{-1} \right) = 
\textup{det}_{\lceil q \rceil} \left( I + 
\Tilde{J} \vert V \vert^{\frac{1}{2}} \big( 
D_m(b,0) - z \big)^{-1} \vert V 
\vert^{\frac{1}{2}} \right),
$$
see Property \textbf{b)} of Section \ref{s3,1}.

The equivalence \textbf{(ii)} $\Leftrightarrow$ 
\textbf{(iii)} is a direct consequence of 
Property \textbf{c)} of Section \ref{s3,1}.

It only remains to prove \eqref{eq4,39}. 
According to the discussion just after 
\eqref{eq4,38}, for $\mathscr{C}'$ a small 
contour positively oriented containing 
$z_{\widetilde m}(k_0)$ as the unique 
discrete eigenvalue of $D_m(b,V)$, we have
\begin{equation}\label{eq4,40}
\textup{mult} \big( z_{\widetilde m}(k_0) 
\big) = ind_{\mathscr{C}'} f,
\end{equation}
$f$ being the function defined by 
\eqref{eq4,38}. The RHS of \eqref{eq4,40} 
is the index defined by \eqref{eqa,1}, of 
the holomorphic function $f$ with respect to 
$\mathscr{C}'$. Now, \eqref{eq4,39} follows 
directly from the equality
$$
ind_{\mathscr{C}'} f = Ind_{\mathscr{C}} 
\hspace{0.5mm} \Big( I + \mathcal{T}_{V} 
\big( z_{\widetilde m}(\cdot) \big) \Big),
$$
see for instance \cite[(2.6)]{bo} for more 
details. This concludes the proof.
\end{proof}

\section{Study of the (weighted) free resolvent}\label{s4}

We split $\mathcal{T}_{V} \big( z_{\widetilde 
m}(k) \big)$ into a singular part at $k = 0$, 
and an analytic part in $\mathcal{D}_\pm^\ast
(\epsilon)$ which is continuous on 
$\overline{\mathcal{D}_\pm^\ast(\epsilon)}$,
with values in $\sqq \big( L^{2}(\brrr^3) \big)$. 
\\

\noindent
For $z := z_{\pm m}(k)$, set
\begin{equation}\label{eq5,1}
\small{\mathcal{T}_{1}^{V} \big( z_{\pm m}(k) 
\big) := \tilde{J} \vert V \vert^{1/2} 
\left[ p \otimes \mathscr{R} \left( k^{2} (z 
\pm m)^{2} \right) \right] 
\left( \begin{smallmatrix}
   z + m & 0 & 0 & 0\\
   0 & 0 & 0 & 0\\
   0 & 0 & z - m & 0\\
   0 & 0 & 0 & 0
\end{smallmatrix} \right) \vert V \vert^{1/2},}
\end{equation}
\begin{equation}\label{eq5,2}
\begin{split}
\small{\mathcal{T}_{2}^{V} \big( z_{\pm m}(k) 
\big) := \tilde{J} \vert} & \small{V \vert^{1/2} 
\left[ p \otimes (-i\partial_{x_3}) \mathscr{R} 
\left( k^{2} (z \pm m)^{2} \right) \right] 
\left( 
\begin{smallmatrix}
   0 & 0 & 1 & 0\\
   0 & 0 & 0 & 0\\
   1 & 0 & 0 & 0\\
   0 & 0 & 0 & 0
\end{smallmatrix} \right) \vert V \vert^{1/2}} \\
& \small{+ \tilde{J} \vert V \vert^{1/2} \big( D_m(b,0) - z 
\big)^{-1} \textup{\textbf{Q}} \vert V \vert^{1/2}.}
\end{split}
\end{equation}
Then, \eqref{eq4,2} combined with \eqref{eq4,4}
imply that
\begin{equation}\label{eq5,3}
\mathcal{T}_{V} \big( z_{\pm m}(k) \big) = 
\mathcal{T}_{1}^{V} \big( z_{\pm m}(k) \big) + 
\mathcal{T}_{2}^{V} \big( z_{\pm m}(k) \big).
\end{equation}

\begin{rem}\label{r5,1}
\textbf{(i)} For $z = z_m(k)$, we have 
\begin{equation}
\small{\im \big( k(z + m) \big) = 
\frac{2m (1 + \vert k \vert^2) 
\im(k)}{\vert 1 + k^2 \vert^2}.} 
\end{equation}
Therefore, according to the choice 
\eqref{eq2,7} of the complex square root, 
we have respectively 
\begin{equation}\label{eq5,4}
\small{\sqrt{k^2(z + m)^2} = \pm k(z + m) 
\quad for \quad k \in \mathcal{D}_\pm^\ast
(\epsilon).}
\end{equation}

\textbf{(ii)} In the case $z = z_{-m}(k)$,
we have
\begin{equation} 
\small{\im \big( k(z - m) \big) = -\frac{2m 
(1 + \vert k \vert^2) \im(k)}{\vert 1 + k^2 
\vert^2}},
\end{equation}
so that
\begin{equation}\label{eq5,5}
\small{\sqrt{k^2(z - m)^2} = \mp k(z - m) 
\quad for \quad k \in \mathcal{D}_\pm^\ast
(\epsilon).}
\end{equation}
\end{rem}

\noindent
In what follows below, we focus on the study
of the operator $\mathcal{T}_{V} \big( z_m(k) 
\big)$, i.e. near $m$. The same arguments yield 
that of the operator $\mathcal{T}_{V} \big( 
z_{-m}(k) \big)$ associated to $-m$, see Remark 
\ref{r5,2}.
\\

\noindent
Defining $G_\pm$ as the multiplication operators 
by the functions $G_\pm : \brrr \ni x_3 \mapsto 
G^{\pm \frac{1}{2}}(x_3)$, we have
\begin{equation}\label{eq5,6}
\begin{split}
\mathcal{T}_{1}^{V} \big( z_{m}(k) \big) = 
\tilde{J} \vert V \vert^{1/2} G_- 
\Big[ p \otimes & G_+ \mathscr{R} \big( k^{2} (z 
+ m)^{2} \big) G_+ \Big] \\
& \times \left( \begin{smallmatrix}
   z + m & 0 & 0 & 0\\
   0 & 0 & 0 & 0\\
   0 & 0 & z - m & 0\\
   0 & 0 & 0 & 0
\end{smallmatrix} \right) G_- \vert V \vert^{1/2}.
\end{split}
\end{equation}
Item \textbf{(i)} of Remark \ref{r5,1} 
together with \eqref{eq4,5} imply that 
$G_+ \mathscr{R} \left( k^{2} (z + m)^{2} 
\right) G_+$ admits the integral kernel
\begin{equation}\label{eq5,7}
\pm G^{\frac{1}{2}}(x_3) 
\frac{i e^{\pm i k(z+m) \vert x_3 - x_3' 
\vert}}{2 k(z+m)} G^{\frac{1}{2}}(x_3'), \quad k 
\in \mathcal{D}_\pm^\ast(\epsilon).
\end{equation}
Then, from \eqref{eq5,7} we deduce that
\begin{equation}\label{eq5,8}
G_+ \mathscr{R} \left( k^{2} (z + m)^{2} 
\right) G_+ = \pm \frac{1}{k(z + m)}a + b_m(k), 
\quad k \in \mathcal{D}_\pm^\ast(\epsilon),
\end{equation}
where $a : L^{2}(\mathbb{R}) \longrightarrow L^{2}
(\mathbb{R})$ 
is the rank-one operator given by 
\begin{equation}\label{eq5,9}
a(u) := \frac{i}{2} \big\langle u,G_+ \big\rangle 
G_+,
\end{equation}
and $b_m(k)$ is the operator with integral 
kernel
\begin{equation}\label{eq5,10}
\pm G^{\frac{1}{2}}(x_3) i \frac{ \textup{e}^{\pm 
i k (z + m)\vert x_3 - x_3' \vert} - 1}{2 k(z + m)} 
G^{\frac{1}{2}}(x_3').
\end{equation}
Note that $-2ia = c^\ast c$, where
$c : L^{2}(\mathbb{R}) \longrightarrow \mathbb{C}$ 
satisfies $c(u) := \langle u,G_+ \rangle$ and 
$c^{\ast} : \mathbb{C} \longrightarrow 
L^{2}(\mathbb{R})$ verifies $c^{\ast}(\lambda) = 
\lambda G_+$. Therefore, by combining 
\eqref{eq5,8}, \eqref{eq5,9} with \eqref{eq5,10}, 
we get for $k \in \mathcal{D}_\pm^\ast(\epsilon)$
\begin{equation}\label{eq5,11}
\begin{aligned}
p \otimes G_+ \mathscr{R} \left( k^{2} (z + m)^{2} 
\right) G_+ = \pm \frac{i}{2k(z + m)} p 
\otimes c^\ast c + p \otimes s_m(k),
\end{aligned}
\end{equation}
where $s_m(k)$ is the operator acting from 
$G^{\frac{1}{2}}(x_3) L^{2}(\mathbb{R})$ 
to $G^{-\frac{1}{2}}(x_3) L^{2}(\mathbb{R})$ with 
integral kernel
\begin{equation}\label{eq5,12}
\pm \frac{ 1 - \textup{e}^{\pm i k(z + m) \vert 
x_3 - x_3' \vert}}{2 i k(z + m)}.
\end{equation}
In Remark \ref{r5,2}, $s_{-m}(k)$ is the 
corresponding operator with $m$ replaced by 
$-m$ and $\pm$ replaced by $\mp$ in 
\eqref{eq5,12}. Now, putting together 
\eqref{eq5,6} and \eqref{eq5,11}, we get for 
$k \in 
\mathcal{D}_\pm^\ast(\epsilon)$
\begin{equation}\label{eq5,13}
\begin{split} 
\mathcal{T}_{1}^{V} \big( z_{m}(k) \big) & = 
\pm \frac{i\Tilde{J}}{2k(z + m)} \vert V 
\vert^{\frac{1}{2}} G_- (p \otimes c^\ast c) 
\left( \begin{smallmatrix}
   z + m & 0 & 0 & 0 \\
   0 & 0 & 0 & 0 \\
   0 & 0 & z - m & 0 \\
   0 & 0 & 0 & 0
\end{smallmatrix} \right) 
G_- \vert V \vert^{\frac{1}{2}} \\
& + \Tilde{J} \vert V \vert^{\frac{1}{2}} 
G_- p \otimes s_m(k) 
\left( \begin{smallmatrix}
   z + m & 0 & 0 & 0 \\
   0 & 0 & 0 & 0 \\
   0 & 0 & z - m & 0 \\
   0 & 0 & 0 & 0
\end{smallmatrix} \right) G_- 
\vert V \vert^{\frac{1}{2}}.
\end{split}
\end{equation}
Introduce the operator
\begin{equation}\label{eq5,14}
K_{\pm m} := \frac{1}{\sqrt{2}} (p \otimes c) 
\left( \begin{smallmatrix}
   1 - 1_\mp & 0 & 0 & 0 \\
   0 & 0 & 0 & 0 \\
   0 & 0 & 1 - 1_\pm & 0 \\
   0 & 0 & 0 & 0
\end{smallmatrix} \right)
G_{-} \vert V \vert^{\frac{1}{2}}, 
\qquad 1_- = 0, \quad 1_+ = 1.
\end{equation} 
It is well known from \cite[Theorem 2.3]{hal} 
that $p$ admits a continuous integral kernel 
$\mathcal{P}(\xp,\xp')$, $\xp$, 
$\xp' \in \brrr^{2}$. Then, we have
$K_{\pm m} : L^{2}(\brrr^{3}) \longrightarrow 
L^{2}(\brrr^{2})$ with 
\begin{align*}
\begin{split}
(K_{\pm m} \psi)(\xp) = \frac{1}{\sqrt{2}} 
\int_{\brrr^{3}} & {\mathcal P} 
(\xp,\xp^\prime) \\
& \times \left( \begin{smallmatrix}
   1 - 1_\mp & 0 & 0 & 0 \\
   0 & 0 & 0 & 0 \\
   0 & 0 & 1 - 1_\pm & 0 \\
   0 & 0 & 0 & 0
\end{smallmatrix} \right) \vert V \vert^{\frac{1}{2}} (\xp^\prime,
x_3^\prime) \psi (\xp^\prime,x_3^\prime)
d\xp^\prime dx_3^\prime.
\end{split}
\end{align*}
Obviously, the operator
$K_{\pm m}^{\ast} : L^{2}(\brrr^{2}) 
\longrightarrow L^{2}(\brrr^{3})$ satisfies
$$
(K_{\pm m}^{\ast}\varphi)(\xp,x_3) = 
\frac{1}{\sqrt{2}} \vert V \vert^{\frac{1}{2}} 
(\xp,x_3) 
\left( \begin{smallmatrix}
   1 - 1_\mp & 0 & 0 & 0 \\
   0 & 0 & 0 & 0 \\
   0 & 0 & 1 - 1_\pm & 0 \\
   0 & 0 & 0 & 0
\end{smallmatrix} \right)
(p \varphi)(\xp).
$$
Noting that
$K_{\pm m} K_{\pm m}^{\ast} : L^{2}(\brrr^{2}) 
\longrightarrow L^{2}(\brrr^{2})$ 
verifies
\begin{equation}\label{eq5,15}
K_{\pm m} K_{\pm m}^{\ast} = 
\left( \begin{smallmatrix}
   1 - 1_\mp & 0 & 0 & 0 \\
   0 & 0 & 0 & 0 \\
   0 & 0 & 1 - 1_\pm & 0 \\
   0 & 0 & 0 & 0
\end{smallmatrix} \right)
p \textbf{\textup{V}}_{\pm m} p,
\end{equation}
$\textbf{\text{V}}_{\pm m}$ being the 
multiplication operators by the functions
(also noted) $\textbf{\text{V}}_{\pm m}$ 
defined by \eqref{eq2,2}.
Thus, by combining \eqref{eq5,13} and 
\eqref{eq5,14} we obtain for $k \in 
\mathcal{D}_\pm^\ast(\epsilon)$
\begin{equation}\label{eq5,16}
\begin{split}
\mathcal{T}_{1}^{V} \big( z_{m}(k) \big) & = 
\pm \frac{i\Tilde{J}}{k} K_m^\ast K_m + 
i\Tilde{J} k K_{-m}^\ast K_{-m} + \\
& + \Tilde{J} \vert V \vert^{\frac{1}{2}} 
G_- p \otimes s_m(k) 
\left( \begin{smallmatrix}
   z + m & 0 & 0 & 0\\
   0 & 0 & 0 & 0\\
   0 & 0 & z - m & 0\\
   0 & 0 & 0 & 0
\end{smallmatrix} \right) G_- 
\vert V \vert^{\frac{1}{2}}.
\end{split}
\end{equation}
Now, for $\lambda \in \brrr_+^\ast$, we 
define $\big( -\partial_{x_3}^2 - \lambda 
\big)^{-1}$ as the operator with integral kernel
\begin{equation}\label{eq5,17}
\displaystyle I_\lambda(x_3,x_3') := 
\lim_{\delta \downarrow 0} I_{\lambda + i\delta} 
(x_3,x_3') = \frac{ie^{i\sqrt{\lambda} 
\vert x_3 - x_3' \vert}} {2\sqrt{\lambda}}.
\end{equation}
where $I_z(\cdot)$ is given by \eqref{eq4,5}. 
Therefore, it can be proved, using a limiting 
absorption principle, that the operator-valued 
functions $\overline{\mathcal{D}_\pm^\ast(\epsilon)} 
\ni k \mapsto G_+ s_m(k) G_+ \in \sd \big( L^{2}
(\brrr) \big)$ are well defined and continuous 
similarly to \cite[Proposition 4.2]{raik}. 
We thus arrive to the following

\begin{prop}\label{p5,1} 
Let $k \in \mathcal{D}_\pm^\ast(\epsilon)$. 
Then, 
\begin{equation}\label{eq5,18}
\mathcal{T}_{V} \big( z_m(k) \big) = 
\pm \frac{i\Tilde{J}}{k} \mathscr{B}_m
+ \mathscr{A}_m(k), \quad \mathscr{B}_m := 
K_m^\ast K_m,
\end{equation}
where $\mathscr{A}_m(k) \in 
\sqq \big( L^{2}(\brrr^3) \big)$ given by
\begin{align*}
\mathscr{A}_m(k) := i\Tilde{J} k K_{-m}^\ast K_{-m}
+ & \Tilde{J} \vert V \vert^{\frac{1}{2}} G_- p 
\otimes s_m(k) \\
& \times \left( \begin{smallmatrix}
   z + m & 0 & 0 & 0\\
   0 & 0 & 0 & 0\\
   0 & 0 & z - m & 0\\
   0 & 0 & 0 & 0
\end{smallmatrix} \right)
G_- \vert V \vert^{\frac{1}{2}}
 + \mathcal{T}_{2}^{V} \big( z_m(k) \big)
\end{align*}
is holomorphic in $\mathcal{D}_\pm^\ast(\epsilon)$ 
and continuous on $\overline{\mathcal{D}_\pm^\ast
(\epsilon)}$.
\end{prop} 

\begin{rem}\label{r5,2}
\textbf{(i)} Identity \eqref{eq5,15} implies that
for any $r > 0$, we have
\begin{equation}\label{eq5,19}
\textup{Tr} \hspace{0.4mm} \one_{(r,\infty)} 
\left( K_{\pm m}^\ast K_{\pm m} \right) 
= \textup{Tr} \hspace{0.4mm} \one_{(r,\infty)} 
\left( K_{\pm m} K_{\pm m}^\ast \right)
= \textup{Tr} \hspace{0.4mm} \one_{(r,\infty)} 
\big( p \textbf{\textup{V}}_{\pm m} p \big).
\end{equation}

\textbf{(ii)} For $V$ verifying Assumption $2.1$,
Proposition \ref{p5,1} holds with $\Tilde{J}$ 
replaced by $J\Phi$, $J := sign(W)$, and in 
\eqref{eq5,19} $\textbf{\textup{V}}_{\pm m}$ 
replaced by $\textbf{\textup{W}}_{\pm m}$.

\textbf{(iii)} Near $-m$, take in account 
item \textbf{(ii)} of Remark \ref{r5,1}, Proposition 
\ref{p5,1} holds with 
\begin{equation}\label{eq5,20}
\mathcal{T}_{V} \big( z_{-m}(k) \big) = 
\mp \frac{i\Tilde{J}}{k} \mathscr{B}_{-m}
+ \mathscr{A}_{-m}(k), \quad \mathscr{B}_{-m} := 
K_{-m}^\ast K_{-m},
\end{equation}
and
\begin{align*}
\mathscr{A}_{-m}(k) := i\Tilde{J} k K_{m}^\ast K_{m}
+ & \Tilde{J} \vert V \vert^{\frac{1}{2}} G_- p 
\otimes s_{-m}(k) \\
& \times \left( \begin{smallmatrix}
   z + m & 0 & 0 & 0\\
   0 & 0 & 0 & 0\\
   0 & 0 & z - m & 0\\
   0 & 0 & 0 & 0
\end{smallmatrix} \right)
G_- \vert V \vert^{\frac{1}{2}}
 + \mathcal{T}_{2}^{V} \big( z_{-m}(k) \big).
\end{align*}
\end{rem}

\section{Proof of the main results}\label{s5}

\subsection{Proof of Theorem \ref{t2,1}}\label{s6}

It suffices to prove that both sums in the LHS of 
\eqref{eq2,15} are bounded by the RHS. We only give 
the proof for the first sum. For the second 
one, the estimate follows similarly by using
Remark \ref{r2,1}-\textbf{(iii),(iv)}, Proposition 
\ref{p5,1}, and Remark \ref{r5,2}-\textbf{(i),(iii)}.
\\

\noindent
In what follows below,
$$
N \big( D_m(b,V) \big) := \big\lbrace \langle 
D_m(b,V)f,f \rangle : f \in Dom \big( D_m(b,V) \big), 
\Vert f \Vert_{L^{2}} = 1 \big\rbrace
$$ 
denotes the numerical range of the operator 
$D_m(b,V)$. It satisfies the inclusion ${\rm sp}\, 
\big( D_m(b,V) \big) \subseteq \overline{N \big( 
D_m(b,V) \big)}$, see e.g. \cite[Lemma 9.3.14]{dav}. 
\\

\noindent
The proof of the theorem uses the following 

\begin{prop}\label{p6,1} 
Let $0 < s_{0} < \epsilon$ be small enough. 
Then, for any $k \in \lbrace 0 < s < \vert k 
\vert < s_{0} \rbrace \cap \mathcal{D}_\pm^\ast
(\epsilon)$, the following properties hold:
 
$\textup{\textbf{(i)}}$ $z_{\pm m}(k) \in 
{\rm sp}\,_{\textup{\textbf{disc}}}^+ \big( 
D_m(b,V) \big)$ near $\pm m$ if and only if 
$k$ is a zero of the determinants
\begin{equation}\label{eq6,1}
\mathscr{D}_{\pm m}(k,s) := \det \big( I + 
\mathscr{K}_{\pm m}(k,s) \big),
\end{equation}
where $\mathscr{K}_{\pm m}(k,s)$ are 
finite-rank operators analytic with respect 
to $k$ such that
\begin{equation}
\small{\textup{rank} \hspace{0.6mm} 
\mathscr{K}_{\pm m}(k,s) = 
\mathcal{O} \Big( \textup{Tr} \hspace{0.4mm} 
\one_{(s,\infty)} \big( p 
\textbf{\textup{V}}_{\pm m} p \big) + 1 \Big)},
\quad \left\Vert \mathscr{K}_{\pm m}(k,s) 
\right\Vert = \mathcal{O} \left( s^{-1} \right),
\end{equation}
uniformly with respect to $s < \vert k \vert < 
s_{0}$. 

$\textup{\textbf{(ii)}}$ Moreover, if 
$z_{\pm m}(k_{0}) \in 
{\rm sp}\,_{\textup{\textbf{disc}}}^+ \big( 
D_m(b,V) \big)$ near $\pm m$, then
\begin{equation}\label{eq6,2}
\textup{mult} \big( z_{\pm m}(k_{0}) \big) = 
Ind_{\mathscr{C}} \hspace{0.5mm} \big( I + 
\mathscr{K}_{\pm m}(\cdot,s) \big) 
= \textup{mult}(k_{0}),
\end{equation}
where $\mathscr{C}$ is chosen as in \eqref{eq4,39},
and $\textup{mult}(k_{0})$ is the multiplicity 
of $k_{0}$ as zero of $\mathscr{D}_{\pm m}
(\cdot,s)$.

$\textup{\textbf{(iii)}}$ If $z_{\pm m}(k)$ 
verifies $\textup{dist} \big( z_{\pm m}(k),
\overline{N \big( D_m(b,V) \big)} \big) > 
\varsigma > 0$, $\varsigma = \mathcal{O}(1)$, then $I + 
\mathscr{K}_{\pm m}(k,s)$ are invertible and 
verify
$
\left\Vert \big( I + \mathscr{K}_{\pm m}(k,s) 
\big)^{-1} \right\Vert = \mathcal{O} \left( 
\varsigma^{-1} \right)
$
uniformly with respect to $s < \vert k \vert < 
s_{0}$.
\end{prop}

%\noindent

\begin{proof}
\textbf{(i)-(ii)} By Proposition \ref{p5,1} 
and item \textbf{(iii)} of Remark 
\ref{r5,2}, the operator-valued functions 
\begin{equation}
k \mapsto \mathscr{A}_{\pm m}(k) \in 
\sqq \big( L^2(\brrr^3) \big)
\end{equation}
are continuous near zero. Then, for $s_{0}$ 
small enough, there exists $\mathscr{A}_{0,\pm m}$ 
finite-rank operators which do not depend on 
$k$, and $\widetilde{\mathscr{A}}_{\pm m}(k) \in 
\sqq \big( L^2(\brrr^3) \big)$ continuous near zero 
satisfying $\Vert \widetilde{\mathscr{A}}_{\pm m}(k) 
\Vert < \frac{1}{4}$ for $\vert k \vert \leq s_{0}$, 
such that
\begin{equation}
\mathscr{A}_{\pm m}(k) = \mathscr{A}_{0,\pm m} + 
\widetilde{\mathscr{A}}_{\pm m}(k).
\end{equation}
Let $\mathscr{B}_{\pm m}$ be the operators defined
respectively by \eqref{eq5,18} and \eqref{eq5,20}.
Then, with the help of the decomposition 
\begin{equation}\label{eq6,3}
\mathscr{B}_{\pm m} = \mathscr{B}_{\pm m} 
\one_{[0,\frac{1}{2}s]} (\mathscr{B}_{\pm m}) 
+ \mathscr{B}_{\pm m} \one_{]\frac{1}{2}s,\infty[} 
(\mathscr{B}_{\pm m}),
\end{equation}
and using that $\left\Vert \pm \frac{i\Tilde{J}}
{k} \mathscr{B}_{\pm m} \one_{[0,\frac{1}{2}s]} 
(\mathscr{B}_{\pm m}) + 
\widetilde{\mathscr{A}}_{\pm m}(k) \right\Vert < 
\frac{3}{4}$ for $0 < s < \vert k \vert < s_{0}$, 
we obtain
\begin{equation}\label{eq6,4}
\begin{split}
\Big( I + \mathcal{T}_{V}\big( 
z_{\pm m}(k) \big) \Big) = \big( I & + 
\mathscr{K}_{\pm m}(k,s) \big) \\
& \times \left( I \pm 
\frac{i\Tilde{J}}{k} 
\mathscr{B}_{\pm m} \one_{[0,\frac{1}{2}s]} 
(\mathscr{B}_{\pm m}) + 
\widetilde{\mathscr{A}}_{\pm m}(k) \right),
\end{split}
\end{equation}
where
\begin{equation}\label{eq6,5}
\begin{split}
\mathscr{K}_{\pm m}(k,s) := \Bigg( \pm 
\frac{i\Tilde{J}}{k} & \mathscr{B}_{\pm m} 
\one_{]\frac{1}{2}s,\infty[} (\mathscr{B}_{\pm m}) 
+ \mathscr{A}_{0,\pm m} \Bigg) \\
& \left( I \pm 
\frac{i\Tilde{J}}{k} \mathscr{B}_{\pm m} 
\one_{[0,\frac{1}{2}s]} (\mathscr{B}_{\pm m}) 
+ \widetilde{\mathscr{A}}_{\pm m}(k) \right)^{-1}.
\end{split}
\end{equation}
Observe that $\mathscr{K}_{\pm m}(k,s)$ are 
finite-rank operators with ranks of order
\begin{equation}
\small{\mathcal{O} \Big( \textup{Tr} \hspace{0.4mm} 
\one_{(\frac{1}{2}s,\infty)} (\mathscr{B}_{\pm m}) 
+ 1 \Big) = \mathcal{O} \Big( \textup{Tr} 
\hspace{0.4mm} \one_{(s,\infty)} 
\big( p \textbf{\textup{V}}_{\pm m} p \big) + 1 
\Big)},
\end{equation}
according to \eqref{eq5,19}. Moreover, their norms 
are of order
$\mathcal{O} \big( \vert k \vert^{-1} \big) = 
\mathcal{O} \big( s^{-1} \big)$. Since from 
above we know that for $0 < s < \vert k \vert < 
s_{0}$ we have the bound $\Vert \pm \frac{i\Tilde{J}}{k} 
\mathscr{B}_{\pm m} \one_{[0,\frac{1}{2}s]} 
(\mathscr{B}_{\pm m}) + 
\widetilde{\mathscr{A}}_{\pm m}(k)\Vert < 
\frac{3}{4} < 1$, then we obtain
\begin{equation}
\small{Ind_{\mathscr{C}} \hspace{0.5mm} \left(I 
\pm \frac{i\Tilde{J}}{k} \mathscr{B}_{\pm m} 
\one_{[0,\frac{1}{2}s]} (\mathscr{B}_{\pm m}) + 
\widetilde{\mathscr{A}}(k) \right) = 0}
\end{equation}
from \cite[Theorem 4.4.3]{goh}. Therefore, 
equalities \eqref{eq6,2} follow by applying to 
\eqref{eq6,4} the properties of the index 
of a finite meromorphic operator-valued 
function recalled in the Appendix B. Proposition 
\ref{p4,1} together with \eqref{eq6,4} 
imply that $z_{\pm m}(k)$ belongs to 
$\sigma_{\textup{\textbf{disc}}}^+ \big( 
D_m(b,V) \big)$ near $\pm m$ if and only if 
$k$ is a zero of the determinants
$\mathscr{D}_{\pm m}(k,s)$ defined by 
\eqref{eq6,1}.
\\

\textbf{(iii)} From \eqref{eq6,4}, we deduce 
that
\begin{equation}\label{eq6,6}
\begin{split}
I + \mathscr{K}_{\pm m}(k,s) = \Big( 
I & + \mathcal{T}_{V}\big( z_{\pm m}(k) \big) 
\Big)  \\
& \times \left( I + \frac{\Tilde{J}}{k} 
\mathscr{B}_{\pm m} \one_{[0,\frac{1}{2}s]} 
(\mathscr{B}_{\pm m}) + 
\widetilde{\mathscr{A}}_{\pm m}(k) \right)^{-1},
\end{split}
\end{equation}
for $0 < s < \vert k \vert < s_{0}$.
With the help of the resolvent equation, it can
be shown that
\begin{equation}
\begin{split}
\Big( I + & \Tilde{J} \vert V \vert^{1/2} 
\big( D_m(b,0) - z \big)^{-1} \vert V \vert^{1/2} 
\Big) \\
& \times \Big( I - \Tilde{J} \vert V \vert^{1/2} 
\big( D_m(b,V) - z \big)^{-1} \vert V \vert^{1/2} 
\Big) = I.
\end{split}
\end{equation}
Then, for $z_{\pm m}(k) \in \rho \big( D_m(b,V) \big)$, 
obviously we have
\begin{equation}
\Big( I + \mathcal{T}_{V} \big( z_{\pm m}(k) \big) 
\Big)^{-1} = I - \Tilde{J} \vert V \vert^{1/2} 
\big( D_m(b,V) - z_{\pm m}(k) 
\big)^{-1} \vert V \vert^{1/2}.
\end{equation}
This together with \eqref{eq6,6} imply the 
invertibility of $I + \mathscr{K}_{\pm m}(k,s)$ 
for $0 < s < \vert k \vert < s_{0}$, and according
to \cite[Lemma 9.3.14]{dav} its verifies
\begin{align*}
\left\Vert \big( I + \mathscr{K}_{\pm m}(k,s) 
\big)^{-1} \right\Vert & = \mathcal{O} \Big( 1 + 
\left\Vert \vert V \vert^{1/2} \big( D_m(b,V) - 
z_{\pm m}(k) \big)^{-1} \vert W \vert^{1/2} 
\right\Vert \Big) \\
& = \mathcal{O} \left( 1 + \textup{dist} \big( 
z_{\pm m}(k),\overline{N \big( D_m(b,V) \big)} 
\big)^{-1} \right) \\
& = \mathcal{O} \left( \varsigma^{-1} \right),
\end{align*}
for $\textup{dist} \big( z_{\pm m}(k),
\overline{N \big( D_m(b,V) \big)} \big) > 
\varsigma > 0$, $\varsigma = \mathcal{O}(1)$. This 
completes the proof.
\end{proof}

\noindent
\textit{End of the proof of Theorem \ref{t2,1}:} 
Now, from item \textbf{(i)} of Proposition \ref{p6,1}, 
we obtain for $0 < s < \vert k \vert < s_{0}$
\begin{equation}\label{eq6,7}
\begin{aligned}
\mathscr{D}_{\pm m}(k,s) & = 
\prod_{j=1}^{\mathcal{O} \big( \textup{Tr} 
\hspace{0.4mm} \one_{(s,\infty)} 
(p \textbf{\textup{V}}_{\pm m} p) + 1 \big)} 
\big{(} 1 + \lambda_{j,\pm m}(k,s) \big{)} \\
& = \mathcal{O}(1) \hspace{0.5mm} \textup{exp} 
\hspace{0.5mm} \Big( \mathcal{O} \big( 
\textup{Tr} \hspace{0.4mm} \one_{(s,\infty)} 
\big( p \textbf{\textup{V}}_{\pm m} p \big) 
+ 1 \big) \vert \ln s \vert \Big),
\end{aligned}
\end{equation}
where the $\lambda_{j,\pm m}(k,s)$ are the 
eigenvalues of $\mathscr{K}_{\pm m} := 
\mathscr{K}_{\pm m}(k,s)$ which satisfy 
$\vert \lambda_{j,\pm m}(k,s) \vert = 
\mathcal{O} \left( s^{-1} \right)$ so that
$\ln \big \vert 1 + \lambda_{j,\pm m}(k,s) \big 
\vert = \mathcal{O} \big( \vert \ln s \vert \big)$ 
(for $s$ small enough). Otherwise, 
we have for $0 < s < \vert k \vert < s_{0}$
$$
\mathscr{D}_{\pm m}(k,s)^{-1} = \det \big( 
I + \mathscr{K}_{\pm m} \big)^{-1} = \det 
\left( I - \mathscr{K}_{\pm m} ( I + 
\mathscr{K}_{\pm m})^{-1} \right)
$$
if $\textup{dist} \big( z_{\pm m}(k),\overline{N 
\big( D_m(b,V) \big)} \big) > \varsigma > 0$. 
Then, as in \eqref{eq6,7}, it can be shown 
that
\begin{equation}\label{eq6,8}
\small{\vert \mathscr{D}_{\pm m}(k,s) \vert 
\geq C \hspace{0.5mm} \textup{exp} \hspace{0.5mm} 
\Big( - C \big( \textup{Tr} \hspace{0.4mm} 
\one_{(s,\infty)} 
\big( p \textbf{\textup{V}}_{\pm m} p \big) + 
1 \big) \big( \vert \textup{ln} \hspace{0.5mm} 
\varsigma \vert + \vert \textup{ln} \hspace{0.5mm} 
s \vert \big) \Big)},
\end{equation}
so that for $ms^2 < \varsigma < 16ms^2$, $0 < s \ll 1$, we obtain
\begin{equation}\label{eq5,70}
- \ln \, \vert \mathscr{D}_{\pm m}(k,s) \vert \leq
C \, \textup{Tr} \hspace{0.4mm} \one_{(s,\infty)} 
\big( p \textbf{\textup{V}}_{\pm m} p \big) \vert \textup{ln} \hspace{0.5mm} s \vert
+ \mathcal{O}(1).
\end{equation}
To conclude, we need the following Jensen 
lemma (see for instance \cite[Lemma 6]{bon} for
a simple proof).

\begin{lem}\label{la,1} 
Let $\Delta$ be a simply connected 
sub-domain of $\mathbb{C}$ and let $g$ be holomorphic 
in $\Delta$ with continuous extension 
to $\overline{\Delta}$. Assume that there exists 
$\lambda_{0} \in \Delta$ such that $g(\lambda_{0}) \neq 0$ 
and $g(\lambda) \neq 0$ for $\lambda\in \partial \Delta$
(the boundary of $\Delta$). 
Let $\lambda_{1}, \lambda_{2}, \ldots, \lambda_{N} \in \Delta$ 
be the zeros of $g$ repeated according to their multiplicity. 
For any domain $\Delta' \Subset \Delta$, there exists 
$C' > 0$ such that $N(\Delta',g)$, the number of zeros 
$\lambda_{j}$ of $g$ contained in $\Delta'$ satisfies
\begin{equation}\label{eqa,5}
N(\Delta',g) \leq C' \left( \int_{\partial \Delta} 
\textup{ln} \vert g(\lambda) \vert d\lambda 
- \textup{ln} \vert g(\lambda_{0}) \vert  \right).
\end{equation}
\end{lem}

\noindent
Consider the sub-domains 
$$
\widetilde{\Delta}_\pm := 
\left\lbrace \frac{1}{2}r < \vert k \vert < 2r : \vert \re(k) \vert > \sqrt{\nu} : 
\vert \im(k) \vert > \sqrt{\nu} : 0 < \nu \ll 1\right\rbrace
\cap \mathcal{D}_\pm^\ast(\epsilon),
$$ 
with $0 < r < \sqrt{\frac{\Vert V \Vert(1 - \gamma)}{2m}} < \sqrt{\frac{3}{2}}r$, 
$\gamma = \frac{1}{2}$. Now, observe that the numerical range of $D_m(b,V)$ verifies
\begin{equation}
N \big( D_m(b,V) \big) \subseteq \big\lbrace z \in \bc : \vert \im(z) 
\vert \leq \Vert V \Vert \big\rbrace,
\end{equation}
so that there exists some value $k_{0} \in \widetilde{\Delta}_\pm /r$ 
which satisfies 
\begin{equation}
\textup{dist} \big( z_{\pm m}(rk_0),\overline{N \big( D_m(b,V) \big)} 
\big) \geq \varsigma > mr^2, \quad \varsigma < 16mr^2.
\end{equation}
Then, we get that the first sum in the LHS of \eqref{eq2,15} is bounded by 
the RHS by using Lemma \ref{la,1} with the functions
$g = g_{\pm m}(k) := \mathscr{D}_{\pm m}(rk,r)$, $k \in \widetilde{\Delta}_\pm /r$,
together with \eqref{eq6,7} and \eqref{eq5,70}.
This concludes the proof of Theorem \ref{t2,1}.

\subsection{Proof of Theorem \ref{t2,3}}\label{s8}

It will only be given for the case 
${\rm Arg}\, \Phi \in (0,\pi)$. To prove the 
case ${\rm Arg}\, \Phi \in -(0,\pi)$, it 
suffices to argue similarly by replacing 
$k$ by $-k$.
\\

\noindent
Remark \ref{r2,2}-\textbf{(i)}, together with 
Proposition \ref{p5,1} and Remark 
\ref{r5,2}-\textbf{(ii),(iii)}, imply that
\begin{equation}\label{eq8,1}
\mathcal{T}_{\varepsilon V} \big( z_{\pm m}(k) 
\big) = \frac{iJ\varepsilon \Phi}{k} 
\mathscr{B}_{\pm m} + \varepsilon 
\mathscr{A}_{\pm m}(k), \qquad k \in 
\mathcal{D}_\pm^\ast(\epsilon),
\end{equation}
where $\mathscr{B}_{\pm m}$ are positive 
self-adjoint operators which do not depend 
on $k$, and $\mathscr{A}_{\pm m}(k) \in 
\sqq \big( L^{2}(\brrr^3) \big)$ are 
holomorphic in $\mathcal{D}_{\pm}^\ast(\epsilon)$ 
and continuous on 
$\overline{\mathcal{D}_{\pm}^\ast(\epsilon)}$. 
Noting that
\begin{equation}
I + \frac{iJ \varepsilon \Phi}{k} 
\mathscr{B}_{\pm m} = \frac{iJ\Phi}{k} 
(\varepsilon \mathscr{B}_{\pm m} - iJk 
\Phi^{-1}),
\end{equation}
it is easy to see that $I + 
\frac{iJ \varepsilon \Phi}{k}\mathscr{B}_{\pm m}$ 
are invertible whenever $iJk \Phi^{-1} \notin 
{\rm sp}\, (\varepsilon \mathscr{B}_{\pm m})$.
Moreover, we have
\begin{equation}\label{eq7,2}
\small{\left\Vert \left( I + \frac{iJ \varepsilon 
\Phi}{k} \mathscr{B}_{\pm m} \right)^{-1} 
\right\Vert \leq \frac{\vert k \Phi^{-1} \vert}{\sqrt{\big( 
J\im(k\Phi^{-1}) \big)_+^2 + \vert \re(k\Phi^{-1}) 
\vert^2}}}.
\end{equation}
Therefore,
\begin{equation}\label{eq8,2}
\small{\left\Vert \left( I + \frac{iJ\varepsilon 
\Phi}{k} \mathscr{B}_{\pm m} \right)^{-1} 
\right\Vert \leq \sqrt{1 + \delta^{-2}}}
\end{equation}
for $k \in \Phi \mathcal{C}_\delta(J)$, uniformly 
with respect to $0 < \vert k \vert < r_0$. 
Then, we deduce from \eqref{eq8,1} that
\begin{equation}\label{eq8,3}
\small{I + \mathcal{T}_{\varepsilon V} 
\big( z_{\pm m}(k) \big) = \big( I + A_{\pm m}(k) 
\big) \left( I + \frac{iJ\varepsilon\Phi}{k} 
\mathscr{B}_{\pm m} \right),}
\end{equation}
where
\begin{equation}\label{eq8,4}
\small{A_{\pm m}(k) := \varepsilon \mathscr{A}_{\pm m} 
(k)\left( I + \frac{iJ\varepsilon \Phi}{k} 
\mathscr{B}_{\pm m} \right)^{-1}} \in \sqq 
\big( L^{2}(\brrr^3) \big).
\end{equation}
Now, by exploiting the continuity of 
$\mathscr{A}_{\pm m}(k) \in \sqq \big( L^{2}(\brrr^3) 
\big)$ near $k = 0$, it can be proved that 
$\Vert \mathscr{A}_{\pm m}(k) \Vert \leq C$ for some 
$C > 0$ constant (not depending on $k$). This together 
with \eqref{eq8,2} and \eqref{eq8,4} imply clearly 
the invertibility of 
$I + \mathcal{T}_{\varepsilon V} \big( z_{\pm m}(k) 
\big)$ for $k \in \Phi \mathcal{C}_\delta(J)$
and $\varepsilon < \big( C \sqrt{1 + \delta^{-2}} 
\big)^{-1}$. Thus, $z_{\pm m}(k)$ is not a discrete 
eigenvalue near $\pm m$.

\subsection{Proof of Theorem \ref{t2,4}}

Denote $(\mu_j^{\pm m})_j$ the sequences of 
the decreasing non-zero eigenvalues of 
$p\textbf{\textup{W}}_{\pm m}p$ taking 
into account the multiplicity. If Assumption 
$2.1$ holds, it can be proved that there 
exists a constant $\nu_{\pm m} > 0$ such 
that
\begin{equation}\label{eq8,16}
\# \big\lbrace j : \mu_j^{\pm m} - 
\mu_{j+1}^{\pm m} > \nu_{\pm m} \mu_j^{\pm m} 
\big\rbrace = \infty.
\end{equation}
Since $\mathscr{B}_{\pm m}$ and 
$p\textbf{\textup{W}}_{\pm m}p$ have the 
same non-zero eigenvalues, then, there exists 
a decreasing sequence $(r_\ell^{\pm m})_{\ell}$, 
$r_\ell^{\pm m} \searrow 0$
with $r_\ell^{\pm m} > 0$, such that
\begin{equation}\label{eq8,17}
\textup{dist} \big( r_\ell^{\pm m},
{\rm sp}\, (\mathscr{B}_{\pm m}) \big) \geq 
\frac{\nu r_\ell^{\pm m}}{2}, \quad \ell \in \bn.
\end{equation}
Furthermore, there exists paths
$\widetilde{\Sigma}_\ell^{\pm m} := \partial 
\Lambda_\ell^{\pm m}$ with
\begin{equation}\label{eq8,181}
\small{\Lambda_\ell^{\pm m} := \big\lbrace \Tilde{k} 
\in \bc : 0 < \vert \Tilde{k} \vert < r_0 : 
\vert \im(\Tilde{k}) \vert \leq \delta \re(\Tilde{k}) 
: r_{\ell + 1}^{\pm m} \leq \re(\Tilde{k}) 
\leq r_\ell^{\pm m} 
\big\rbrace},
\end{equation}
(see Figure 5.1) enclosing the eigenvalues 
of $\mathscr{B}_{\pm m}$ lying in
$[r_{\ell + 1}^{\pm m},r_\ell^{\pm m}]$.

\begin{figure}[h]\label{fig 4}
\begin{center}
\tikzstyle{+grisEncadre}=[dashed]
\tikzstyle{blancEncadre}=[fill=white!100]
\tikzstyle{grisEncadre}=[densely dotted]
\tikzstyle{dEncadre}=[dotted]

\begin{tikzpicture}[scale=1]

\draw [->] (0,-2) -- (0,4);
\draw [->] (-1.5,0) -- (6,0);

\draw (1.9,-0.95) -- (1.9,0.95) -- (3.5,1.75) -- (3.5,-1.75) -- cycle;

\draw [grisEncadre] (0,0) -- (1.9,0.95);
\draw [grisEncadre] (3.5,1.75) -- (5.5,2.75);

\node at (5.7,3) {$\im(\Tilde{k}) = \delta \hspace*{0.08cm} 
\re(\Tilde{k})$};
\node at (2.3,-1.5) {$\widetilde{\Sigma}_{\ell}^{\pm m}$};

\draw [->] (1.5,-0.5) -- (1.86,-0.05);
\node at (1.4,-0.7) {$r_{\ell + 1}^{\pm m}$};

\draw [->] (3.9,-0.5) -- (3.54,-0.04);
\node at (4,-0.7) {$r_{\ell}^{\pm m}$};

\node at (3.3,0) {\tiny{$\bullet$}};
\node at (3.3,0.2) {\tiny{$\mu_{j}^{\pm m}$}};

\node at (3.7,0) {\tiny{$\bullet$}};
\node at (3.9,0.2) {\tiny{$\mu_{j-1}^{\pm m}$}};

\node at (1.6,0) {\tiny{$\bullet$}};
\node at (2.2,0) {\tiny{$\bullet$}};

\node at (2.6,0) {\tiny{$\bullet$}};
\node at (2.7,0.2) {\tiny{$\mu_{j+1}^{\pm m}$}};

\node at (1.2,0) {\tiny{$\bullet$}};
\node at (1,0) {\tiny{$\bullet$}};
\node at (0.5,0) {\tiny{$\bullet$}};

\node at (4.3,0) {\tiny{$\bullet$}};
\node at (4.6,0) {\tiny{$\bullet$}};

\end{tikzpicture}
\caption{Representation of the paths 
$\widetilde{\Sigma}_\ell^{\pm m}  
= \partial \Lambda_\ell^{\pm m}$, 
$\ell \in \bn$.}
\end{center}
\end{figure}
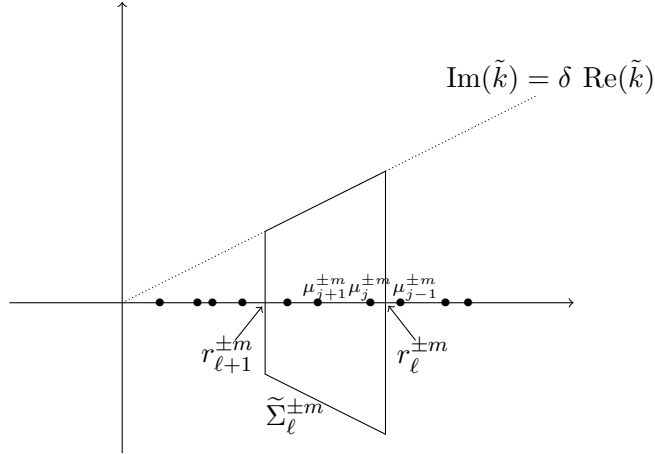

\noindent
Clearly, the operators $\Tilde{k} - 
\mathscr{B}_{\pm m}$ are invertible for 
$\Tilde{k} \in 
\widetilde{\Sigma}_{\ell}^{\pm m}$. Moreover, 
it can be easily checked that
\begin{equation}\label{eq8,18}
\big\Vert (\Tilde{k} - \mathscr{B}_{\pm m})^{-1} 
\big\Vert \leq \frac{\max \left( \delta^{-1}\sqrt{1 
+ \delta^2},(\nu/2)^{-1}\sqrt{1 + \delta^2} \right)}{\vert \Tilde{k} \vert}.
\end{equation}
Set $\Sigma_{\ell}^{\pm m} := -iJ\varepsilon 
\Phi \widetilde{\Sigma}_{\ell}^{\pm m}$.
The construction of the paths $\Sigma_{\ell}^{\pm m}$ 
together with \eqref{eq8,18} imply that
$I + \frac{iJ\varepsilon \Phi}{k} 
\mathscr{B}_{\pm m}$ are invertible for $k \in 
\Sigma_{\ell}^{\pm m}$ with
\begin{equation}\label{eq8,20}
\left\Vert \left( I + \frac{iJ\varepsilon 
\Phi}{k} \mathscr{B}_{\pm m} \right)^{-1} 
\right\Vert \leq \max \left( \delta^{-1}\sqrt{1 
+ \delta^2},(\nu/2)^{-1}\sqrt{1 + \delta^2} \right).
\end{equation}
Hence, we have
\begin{equation}\label{eq8,21}
\begin{split}
\small{I + \frac{iJ\varepsilon \Phi}{k} 
\mathscr{B}_{\pm m} + \varepsilon 
\mathscr{A}_{\pm m}(k) =} & \small{\bigg( I + 
\varepsilon \mathscr{A}_{\pm m}(k) \bigg( I + 
\frac{iJ \varepsilon \Phi}{k} 
\mathscr{B}_{\pm m} \bigg)^{-1} \bigg)} \\
& \small{\times \bigg( I + \frac{iJ\varepsilon 
\Phi}{k} \mathscr{B}_{\pm m} \bigg)},
\end{split}
\end{equation}
for $k \in \Sigma_{\ell}^{\pm m}$. Now, if we 
choose $0 < \varepsilon \leq \varepsilon_0$ small
enough and use Property \textbf{e)} given by 
\eqref{eq3,5}, we get for any $k \in 
\Sigma_{\ell}^{\pm m}$
\begin{equation}\label{eq8,22}
\small{\left\vert \textup{det}_{\lceil q \rceil} 
\left[ I + \varepsilon \mathscr{A}_{\pm m}(k) 
\left( I + \frac{iJ \varepsilon \Phi}{k} 
\mathscr{B}_{\pm m} \right)^{-1} \right] - 1 
\right\vert < 1.}
\end{equation}
Therefore, from the Rouché Theorem, we know that 
the number of zeros of $\textup{det}_{\lceil q 
\rceil} \big( I + \frac{iJ\varepsilon \Phi}{k} 
\mathscr{B}_{\pm m} + \varepsilon \mathscr{A}_{\pm m}
(k) \big)$ enclosed in $\big\lbrace z_{\pm m}(k) \in 
{\rm sp}\,_{\textup{\textbf{disc}}}^+ \big( D_m(b,
\varepsilon V) \big) : k \in \Lambda_{\ell}^{\pm m} 
\big\rbrace$ taking into account the multiplicity, 
coincides with that of 
$\textup{det}_{\lceil q \rceil} \big( I + \frac{iJ 
\varepsilon \Phi}{k} \mathscr{B}_{\pm m} \big)$
enclosed in $\big\lbrace z_{\pm m}(k) \in 
{\rm sp}\,_{\textup{\textbf{disc}}}^+ \big( D_m(b,
\varepsilon V) \big) : k \in \Lambda_{\ell}^{\pm m} 
\big\rbrace$ taking 
into account the multiplicity. The number of zeros 
of $\textup{det}_{\lceil q \rceil} \big( I + \frac{iJ 
\varepsilon \Phi}{k} \mathscr{B}_{\pm m} \big)$
enclosed in $\big\lbrace z_{\pm m}(k) \in 
{\rm sp}\,_{\textup{\textbf{disc}}}^+ \big( D_m(b,
\varepsilon V) \big) : k \in \Lambda_{\ell}^{\pm m} 
\big\rbrace$ taking into account the multiplicity 
is equal to $\textup{Tr} \hspace{0.4mm} 
\one_{[r_{\ell +1}^{\pm m},r_\ell^{\pm m}]} \big( 
p \textbf{\textup{W}}_{\pm m}p \big)$. The 
zeros of $\textup{det}_{\lceil q \rceil} \big( I + 
\frac{iJ\varepsilon \Phi}{k} \mathscr{B}_{\pm m} 
+ \varepsilon \mathscr{A}_{\pm m}(k) \big)$ are the 
discrete eigenvalues of $D_m(b,\varepsilon V)$ near 
$\pm m$ taking into account the multiplicity. Then, 
this together with Proposition \ref{p4,1} and 
Property \eqref{eqa,3} applied to \eqref{eq8,21} 
give estimate \eqref{eq2,30}. Since the sequences 
$(r_\ell^{\pm m})_\ell$ are infinite tending to zero, 
then the infiniteness of the number of the discrete 
eigenvalues claimed follows, which completes the 
proof of Theorem \ref{t2,4}.

%\newpage

\section{Appendix A: Reminder on Schatten-von Neumann ideals and regularized determinants}\label{s3,1}

%\subsection{Reminder on Schatten-von Neumann ideals and regularized determinants}\label{s3,1}

Consider a separable Hilbert space $\mathscr{H}$. Let
$\sinf(\mathscr{H})$ denote the set of compact linear 
operators on $\mathscr{H}$, and $s_k(T)$ be the $k$-th 
singular value of $T \in \sinf(\mathscr{H})$. 
For $q \in [1,+\infty)$, the Schatten-von Neumann classes
are defined by 
\begin{equation}\label{eq3,1}
\sqq(\mathscr{H}) := \Big\lbrace T \in \sinf(\mathscr{H}) : 
\Vert T \Vert^q_\sqq := \sum_k s_k(T)^q < +\infty \Big\rbrace.
\end{equation}
When no confusion can arise, we write $\sqq$ for simplicity.
If $T \in \sqq$ with $\lceil q \rceil := \min \big\lbrace 
n \in \mathbb{N} : n \geq q \big\rbrace$, the $q$-regularized 
determinant is defined by
\begin{equation}\label{eq3,2}
\small{\textup{det}_{\lceil q \rceil} (I - T)
 := \prod_{\mu \hspace*{0.1cm} \in \hspace*{0.1cm} \sigma (T)} 
 \left[ (1 - \mu) \exp \left( \sum_{k=1}^{\lceil q \rceil-1} 
 \frac{\mu^{k}}{k} \right) \right]}.
\end{equation}
In particular, when $q = 1$, to simplify the notation 
we set
\begin{equation}
\textup{det} (I - T) := \textup{det}_{\lceil 1 \rceil}
(I - T). 
\end{equation}
Let us give (see for instance \cite{simo}) some 
elementary useful properties on this determinant.

\textbf{a)} We have $\textup{det}_{\lceil q \rceil} (I) = 1$.

\textbf{b)} If $A$, $B \in \mathscr{L} (\mathscr{H})$ with
$AB$ and $BA$ lying in $\sqq$, then
$\textup{det}_{\lceil q \rceil} (I - AB)
= \textup{det}_{\lceil q \rceil} (I - BA)$. Here, $\mathscr{L} 
(\mathscr{H})$ is the set of bounded linear operators 
on $\mathscr{H}$.

\textbf{c)} $I - T$ is an invertible operator if and only if
$\textup{det}_{\lceil q \rceil} (I - T) \neq 0$.

\textbf{d)} If $T : \Omega \longrightarrow \sqq$ is a 
holomorphic operator-valued function on a domain $\Omega$, 
then so is $\textup{det}_{\lceil q \rceil} \big( I - T(\cdot) 
\big)$ on $\Omega$.

\textbf{e)} As function on $\sqq$, $\textup{det}_{\lceil q 
\rceil} (I - T)$ is Lipschitz uniformly on balls. This means
that
\begin{equation}\label{eq3,5}
\begin{split}
\big\vert \textup{det}_{\lceil q \rceil} (I - T_1) & -
\textup{det}_{\lceil q \rceil} (I - T_2) \big\vert
\leq \Vert T_1 - T_2 \Vert_\sqq \\
& \times \exp \left( \Gamma_q \big( \Vert T_1 \Vert_\sqq + 
\Vert T_2 \Vert_\sqq + 1 \big)^{\lceil q \rceil} \right),
\end{split}
\end{equation}
\cite[Theorem 6.5]{simo}.

\section{Appendix B: On the index of a finite meromorphic operator-valued function}\label{sa,1}

We refer for instance to \cite{gohb} for the 
definition of a finite meromorphic 
operator-valued function.
\\

\noindent
Let $f : \Omega \rightarrow \bc $ be a holomorphic 
function in a vicinity of a contour $\mathscr{C}$ 
positively oriented. Then, it index with respect 
to $\mathscr{C}$ is given by
\begin{equation}\label{eqa,1}
ind_{\mathscr{C}} \hspace{0.5mm} f 
:= \frac{1}{2i\pi} \int_{\mathscr{C}}
\frac{f'(z)}{f(z)} dz.
\end{equation}
Observe that if 
%$f$ is holomorphic in a domain $\Omega$ such that 
$\partial \Omega = \mathscr{C}$, 
then $\textup{ind}_{\mathscr{C}} \hspace{0.5mm} f$ 
is equal to the number of zeros of $f$ lying in 
$\Omega$ taking into account their multiplicity 
(by the residues theorem).
\\

\noindent
If $D \subseteq \mathbb{C}$ is a connected domain, 
$Z \subset D$ a closed pure point subset, 
$A : \overline{D} \backslash Z \longrightarrow 
\textup{GL}(\mathscr{H})$ a finite meromorphic operator-valued 
function which is Fredholm at each point of $Z$, 
the index of $A$ with respect to the contour 
$\partial \Omega$ is given by 
\begin{equation}\label{eqa,2}
\small{Ind_{\partial \Omega} \hspace{0.5mm} A := 
\frac{1}{2i\pi} \textup{Tr} \int_{\partial \Omega} A'(z)A(z)^{-1} 
dz = \frac{1}{2i\pi} \textup{Tr} \int_{\partial \Omega} A(z)^{-1} 
A'(z) dz}.
\end{equation} 
We have the following well known properties: 
\begin{equation}\label{eqa,3}
Ind_{\partial \Omega} \hspace{0.5mm} A_{1} A_{2} = 
Ind_{\partial \Omega} \hspace{0.5mm} A_{1} + 
Ind_{\partial \Omega} \hspace{0.5mm} A_{2};
\end{equation} 
if $K(z)$ is of trace-class, then
\begin{equation}\label{eqa,4}
Ind_{\partial \Omega} \hspace{0.5mm} (I+K)= 
ind_{\partial \Omega} \hspace{0.5mm} \det \hspace{0.5mm} (I + K),
\end{equation} 
see for instance to \cite[Chap. 4]{goh} for more details.

%\newpage

\end{document}